\documentclass{scrartcl}

\usepackage{hyperref, color}
\usepackage[utf8]{inputenc}
\usepackage{amsthm, amsmath, amssymb}
\usepackage{mathtools, mathrsfs}
\usepackage[normalem]{ulem}
\usepackage{todonotes}

\numberwithin{equation}{section}
\newtheorem{theorem}{Theorem}[section]
\newtheorem{proposition}{Proposition}[section]

\newtheorem{lem}{Lemma}[section]
\newtheorem{definition}{Definition}[section]

\theoremstyle{definition}
\newtheorem{remark}{Remark}[section]
\newtheorem{rmk}{Example}[section]

\newcommand{\bC}{\mathbb{C}}

\newcommand{\bH}{\mathbb{H}}
\newcommand{\bN}{\mathbb{N}}
\newcommand{\bZ}{\mathbb{Z}}
\newcommand{\bR}{\mathbb{R}}
\newcommand{\bS}{\mathbb{S}}
\newcommand{\cL}{\mathcal{L}}
\newcommand{\cZ}{\mathcal{Z}}

\newcommand{\R}{\mathbb{R}}
\newcommand{\dmu}{d\mu}

\renewcommand{\aa}{\alpha_*}
\newcommand{\mm}{m_*}
\renewcommand{\b}{\mathfrak b}

\newcommand{\flux}[1]{\mathfrak F_{#1}}

\newcommand{\spec}{\sigma}

\newcommand{\nH}{\nabla}
\newcommand{\dH}{d_{\mathrm H}}

\newcommand{\sublap}{\Delta_{\mathbb H}}
\newcommand{\co}{P_A}
\newcommand{\cf}{\mathfrak p_A}
\newcommand{\cfo}{\mathfrak p_0}
\newcommand{\OmegaH}{\Omega_{\mathrm H}}
\newcommand{\eul}{\mathcal{E}}

\DeclareMathOperator{\dom}{dom}
\DeclareMathOperator{\supp}{supp}

\DeclarePairedDelimiter{\abs}{\lvert}{\rvert}
\DeclarePairedDelimiter{\norm}{\lVert}{\rVert}

\definecolor{biagio}{rgb}{0.16, 0.32, 0.75}
\definecolor{dario}{rgb}{0.82, 0.41, 0.12}
\definecolor{vale}{rgb}{0.53, 0.66, 0.42}

\usepackage[normalem]{ulem}
\definecolor{DarkGreen}{rgb}{0,0.5,0.1} 

\newcommand\soutD{\bgroup\markoverwith
{\textcolor{DarkGreen}{\rule[.5ex]{2pt}{1pt}}}\ULon}
\newcommand{\Hm}[1]{\leavevmode{\marginpar{\tiny%
$\hbox to 0mm{\hspace*{-0.5mm}$\leftarrow$\hss}%
\vcenter{\vrule depth 0.1mm height 0.1mm width \the\marginparwidth}%
\hbox to
0mm{\hss$\rightarrow$\hspace*{-0.5mm}}$\\\relax\raggedright #1}}}

\title{Horizontal magnetic fields and improved Hardy inequalities in the Heisenberg group}

\author{
 Biagio Cassano\thanks{Department of Mathematics, Universit\`a degli
  Studi di Bari ``A.~Moro'', via Orabona 4, 70125, Bari, Italy;
\texttt{biagio.cassano@uniba.it}.} 
 \and 
 Valentina Franceschi\thanks{Dipartimento di Matematica Tullio Levi-Civita, Universit\`a di Padova; \texttt{valentina.franceschi@unipd.it}.}
 \and 
 David Krejčiřík
 \thanks{Department of Mathematics, Faculty of Nuclear Sciences and Physical Engineering, Czech Technical University in Prague, Trojanova 13, 12000 Prague 2, Czechia; \texttt{david.krejcirik@fjfi.cvut.cz}.}
 \and 
 Dario Prandi
 \thanks{Université Paris-Saclay, CNRS, CentraleSupélec, Laboratoire des signaux et systèmes, 91190,
Gif-sur-Yvette, France. \texttt{dario.prandi@centralesupelec.fr}.}
}

\date{\vspace{-1em}}

\begin{document}
	\maketitle

\begin{abstract}
In this paper we introduce a notion of magnetic field in the Heisenberg group and we study its influence on spectral properties of the corresponding magnetic (sub-elliptic) Laplacian.
We show that uniform magnetic fields uplift the bottom of the spectrum.
For magnetic fields vanishing at infinity,
including Aharonov--Bohm potentials, we
derive magnetic improvements to a variety of Hardy-type inequalities for the Heisenberg sub-Laplacian.
In particular, we
establish a sub-Riemannian analogue of Laptev and Weidl sub-criticality result for magnetic Laplacians in the plane. 
Instrumental for our argument is the validity of a Hardy-type inequality for the Folland--Stein operator, that we prove in this paper and has an interest on its own.
\end{abstract}			
%


\section{Introduction}
%
It is well known that the (elliptic) Laplacian 
in $\R^n$ with $n \geq 3$ satisfies the Hardy inequality 
(in the sense of quadratic forms 
for self-adjoint realizations of the respective operators in $L^2(\R^n)$)
\begin{equation}\label{Hardy}
  -\Delta \geq \frac{c_n}{\varrho^{2}}
  \qquad \mbox{in} \qquad
  L^2(\R^n) \,,
\end{equation}
where $\varrho(x) := \sqrt{x_1^2 + \dots + x_n^2}$ is the Euclidean distance to the origin of~$\R^n$
and $c_n := (n-2)^2/4$ is a dimensional constant.
Moreover, the inequality is optimal in the sense that
no positive term can be added to the right-hand-side of~\eqref{Hardy},
see \cite[Sec.~8.1]{Weidl_1999a}. 
On the other hand, if $n=1$ or $n=2$,
there exists no positive function~$w$ such that $-\Delta \geq w$.
These properties are termed as the \emph{criticality}
versus the \emph{subcriticality} of the Laplacian
in the low and high dimensions of the Euclidean space, respectively.
These notions naturally extend to the setting of 
more general elliptic operators 
and Riemannian manifolds,
where they coincide with the alternative concepts of 
the \emph{parabolicity/recurrency}
versus the \emph{non-parabolicity/transiency}
(see~\cite{Pinchover_2007} for an overview).

In 1998 Laptev and Weidl demonstrated in \cite{LW99}
that adding a magnetic field to the Laplacian in~$\R^2$
makes the operator subcritical, meaning that there is a Hardy-type inequality.
The result has stimulated an enormous growth of interest in
magnetically induced Hardy-type inequalities for elliptic operators 
with many important applications in quantum mechanics and elsewhere.  
More generally (see \cite{Weidl_1999,Cazacu2016}), 
while the shifted operator $-\Delta - c_n / \varrho^{2}$ 
is critical in $L^2(\R^n)$ for all $n \geq 2$,
it becomes subcritical after adding a magnetic field to the Laplacian.

The purpose of the present paper is to study the influence of magnetic fields on the criticality properties of
the (sub-elliptic) Laplacian in the Heisenberg group~$\mathbb H^1$.
The latter is the foremost example of sub-Riemannian structure on $\mathbb R^3$
formally defined by the completely non-integrable distribution 
$\mathcal D := \operatorname{span}\{X,Y\}$, where 
\begin{equation}\label{fields}
    X := \partial_x - \frac y2\partial_z, \qquad 
    Y := \partial_y + \frac x2\partial_z,
\end{equation}
with $(x,y,z) \in \mathbb R^3$.
The directions of $\mathcal D$ are called \emph{horizontal directions}, and fixing on~$\mathcal D$ the metric for which $\{X,Y\}$ is an orthonormal frame allows to define a distance on $\bH^1$.
The associated Laplacian is 
the sub-elliptic operator $-\Delta := -X^2 - Y^2$, that has been extensively studied in the last fifty years due to its deep connections 
with diverse subjects; see, e.g., \cite{hormanderHypoelliptic1967a,Folland73, follandEstimates1974, JL88, Montgomery1995}.

Similarly to~\eqref{Hardy}, it is known due to
Garofalo and Lanconelli \cite{GarofaloLanconelli1990}
that the optimal Hardy-type inequality
\begin{equation}\label{eq:GL}
  -\Delta \geq \frac{r^2}{\rho^4}
  \qquad \mbox{in} \qquad
  L^2(\mathbb H^1)
\end{equation}
holds, 
where \(\rho(x,y,z):=\sqrt[4]{(x^2+y^2)^2+16z^2}\) 
is the Koranyi norm
and $r(x,y,z):=\sqrt{x^2+y^2}$ is the radial distance to the $z$-axis. 
Notice that $r^2/\rho^4 = |\nabla \rho|^2/\rho^2$, 
where $|{\nabla} \rho|^2 = |X\rho|^2 + |Y\rho|^2$ is the norm of the \emph{horizontal gradient} of $\rho$, associated with the sub-Riemannian structure of $\bH^1$.
Further results on Hardy-type inequalities on the Heisenberg group are presented in  \cite{NZY2001,D'Ambrosio2004,FranceschiPrandi20}. We refer to \cite{GoldsteinKombe2008,GKY2018,Ruzhansky2017,RS17b,CiattiRicciSundari07,CiattiCowlingRicci15,Loiudice18} for extensions to more general Carnot Groups, and to  \cite{ABFP,D'Ambrosio2003,D'Ambrosio2005,DGP2011,FPR20,Kogoj2016}
for extensions to sub-elliptic operators that are not necessarily induced by groups.

One of the main results of the present paper states that,
while the shifted operator $-\Delta - r^2 / \rho^{4}$ is critical in $L^2(\mathbb H^1)$,
it becomes subcritical after adding a magnetic field to the Laplacian.
In this way we establish a sub-Riemannian analogue of the celebrated
result of Laptev and Weidl~\cite{LW99}.
Other functional inequalities in the Heisenberg group 
and their respective magnetic improvements are also 
investigated.

\section{Main results}
%

\subsection{Horizontal magnetic fields on the Heisenberg group}
To state our main results, 
we first need to properly introduce magnetic fields in the Heisenberg group.
The notion of magnetic fields and associated magnetic operators 
are naturally formulated in terms of differential forms.

On a Riemannian manifold~$M$,
a \emph{magnetic field}~$B$ is a closed $2$-form, 
i.e., $dB = 0$ where $d$ is the exterior differential.
The corresponding \emph{magnetic potential}~$A$ 
is a $1$-form such that $dA=B$.
The latter allows one to define the classical 
and quantum dynamics 
under the influence of the magnetic field
as follows.
Classically, given a Hamiltonian $h\in C^\infty(T^*M)$, 
the magnetic Hamiltonian is 
$h_A(p,q):= h(p+A(q),q)$,
where~$q$ and~$p$ are the (generalized) coordinates and momenta on~$M$; 
the magnetic classical trajectories 
are defined via standard Hamiltonian equations.
The magnetic quantum dynamics are then obtained 
by the Schr\"odinger equation with respect to 
an appropriate quantization~$H_A$ of~$h_A$.
In the case of~$h$ being the free Hamiltonian on~$M$,
the quantum magnetic Hamiltonian~$H_A$ 
coincides with the \emph{magnetic Laplacian}
\begin{equation}\label{Laplace}
  -\Delta_A := (-i\nabla + A)^2
  \,,
\end{equation}
where~$\nabla$ is the Riemannian gradient;
or equivalently, $\Delta_A = {\nabla_{\!A}}^2$, 
where $\nabla_{\!A} := \nabla+iA$ is the \emph{magnetic gradient}.
Of course, $-\Delta_0 = -\Delta$.
 
In the context of the Heisenberg group~$\mathbb{H}^1$, 
it is natural to define the magnetic fields in such a way that the corresponding classical magnetic trajectories retain 
the property of being \emph{horizontal}. 
As we will show later in Section~\ref{sec:horizontal}, 
this is not really a requirement inasmuch it is a natural consequence of the definition of classical magnetic trajectories. 
This naturally leads to define magnetic fields in the context of 
the \emph{Rumin complex} \cite{Rumin1994}. 
The latter roughly corresponds to consider vector potentials modulo the contact form~$\omega$ defining~$\mathcal D$ (i.e., $A = A_x dx + A_y dy \mod \omega$)  and magnetic fields as horizontal $2$-forms (i.e., such that $B\wedge \omega = 0$). 
Then, formally as above, the magnetic (sub-)Laplacian in~$\mathbb{H}^1$ 
is given by~\eqref{Laplace} with the only difference
that $\nabla u := (Xu) X + (Yu)Y$ is the now the horizontal gradient. 
Let us remark that the Rumin complex has already been applied to derive Maxwell's Equations in the more general setting of Carnot groups in \cite{FranchiTesi12, FOV13}.

To the best of our knowledge, 
the notion of the Heisenberg sub-Laplacian with a magnetic field has not been studied yet, except for \cite{Xiao2015}. We discuss this paper in Appendix~\ref{app:xiao}, where we raise a crucial criticism of its main result. We also mention \cite{AermarkLaptev11} where the authors investigate improved Hardy inequalities in the Grushin setting in presence of Aharonov--Bohm magnetic fields. However,  the magnetic operator introduced in that work, when interpreted in our setting, corresponds to a magnetic sub-Laplacian with an extra electric potential.

\subsection{Uniform magnetic fields}
In a Riemannian manifold~$M$,
it is well known that the magnetic field has a deep influence
on spectral properties of the magnetic Laplacian~$-\Delta_A$
(when realized as a self-adjoint operator in $L^2(M)$).
Roughly, the magnetic field acts as a repulsive interaction,
which is known as the \emph{diamagnetic effect} in quantum mechanics. 

The diamagnetic effect 
is best seen in the case of 
uniform (or homogeneous) fields in the plane~$\R^2$,
i.e., $B(x,y)=b \, dx\wedge dy$ with $b\in \mathbb R$.
In this case, one has
(see \cite[Corol.~2.5]{Mohamed-Raikov} 
and \cite[Thm.~3]{Montgomery1995} for higher dimensions
and a Riemannian counterpart, respectively)
\begin{equation}\label{Poincare}
  \inf \sigma(-\Delta_A) = |B|
  \qquad \mbox{in} \qquad 
  L^2(\R^2)
  \,,
\end{equation}
where~$A$ is any 1-form in~$\R^2$ such that $B=dA$
and we write $|B|:=|b|$.
Since $\sigma(-\Delta) = [0,\infty)$,
it is clear that any non-trivial uniform magnetic field uplifts the bottom of the spectrum.
This fact particularly implies that the heat semigroup
associated with $-\Delta_A$ in $L^2(\R^2)$ 
admits a faster decay rate once the magnetic field is turned on.
As an immediate consequence of~\eqref{Poincare},
one has the optimal Poincar\'e-type inequality
$-\Delta_A \geq |B|$ in the sense of quadratic forms in $L^2(\R^2)$.
This inequality extends to the case of variable 
magnetic fields of strength bounded from below 
by the positive constant~$|b|$.
 
Our first result is the following generalization
of~\eqref{Poincare} to the Heisenberg group.
\begin{theorem}\label{thm:uniform}
Let $B(x,y,z)=b_1 \, dx\wedge \omega + b_2 \, dy\wedge \omega$
with $b := (b_1,b_2)\in \mathbb R^2$.
Then 
\begin{equation}\label{Poincare.Heisenberg}
  \inf \sigma(-\Delta_A) = c \, |B|^{2/3}
  \qquad \mbox{in} \qquad 
  L^2(\mathbb H^1) \,,
\end{equation}
where~$A$ is any 1-form in~$\mathbb H^1$ such that $B=dA$, 
$|B| := \sqrt{b_1^2 + b_2^2}$, and $c>0$ is a universal constant.
\end{theorem}

The non-linear growth in the strength~$|b|$ 
of the magnetic field in~\eqref{Poincare.Heisenberg}
is ultimately related with the results obtained in \cite{Montgomery1995}, 
see Remark~\ref{rem:montgomery} below.

As above, \eqref{Poincare.Heisenberg} implies 
the optimal Poincar\'e-type inequality
$-\Delta_A \geq c \, |B|^{2/3}$ 
in the sense of quadratic forms in $L^2(\mathbb{H}^1)$.
This is again a non-trivial diamagnetic improvement 
due to the magnetic field,
because $\sigma(-\Delta) = [0,\infty)$
in the Heisenberg case as well.

Note that the results~\eqref{Poincare} and \eqref{Poincare.Heisenberg}
depend on~$B$ only, while they are independent of~$A$.
This is natural and physically expected because 
of the vanishing of the first cohomology group of~$\R^n$. 
More specifically, given any $0$-form $f$,
the operators $-\Delta_A$ and $-\Delta_{A+df}$
are unitarily equivalent, so isospectral.
This is known as \emph{gauge invariance}
of the magnetic field in quantum mechanics.

\subsection{Aharonov--Bohm magnetic potentials}
Interesting and unexpected phenomena appear 
for more complex geometries when the gauge invariance 
does not hold. 
In this setting, the non-exact vector potentials~$A$ 
yielding null magnetic fields 
(i.e., such that $dA =0$ but $A\neq df$ for any smooth function $f$)
are known as \emph{Aharonov--Bohm potentials}.

The simplest example is the punctured plane $\R^2 \setminus \{0\}$ 
with the magnetic potential 
$A_\alpha := \alpha \,d\varphi$ where $\alpha\in \R$ 
and $(\varrho,\varphi) \in (0,\infty) \times \mathbb{S}^1$ are polar coordinates.
It is important to notice that $A_\alpha$ is actually
a strongly singular vector potential;
indeed, in Cartesian coordinates one has 
$A_\alpha = \alpha \, \varrho(x,y)^{-2} \, (x dy - ydx)$.
In particular, $A_\alpha$ is not locally square integrable in~$\R^2$. 

Although classically invisible
(indeed, $dA_\alpha = 0$ in $\R^2 \setminus \{0\}$), 
the Aharonov--Bohm potential~$A_\alpha$ 
still has a strong influence on 
spectral properties of~$-\Delta_{A_\alpha}$,
and therefore on quantum dynamics.
Indeed, it can break the essential self-adjointness of 
the magnetic Laplacian \cite{AdamiTeta,dabrowski1998aharonov}.
Moreover, even in the case of the Friedrichs extension, 
where the spectrum stays unchanged, 
it was observed by Laptev and Weidl in~\cite{LW99} 
that the presence of the Aharonov--Bohm potential  
makes the magnetic Laplacian subcritical.
More specifically, one has the optimal magnetic Hardy inequality
\begin{equation}\label{Hardy.LW}
  -\Delta_{A_\alpha} \ge \frac{d(\alpha, \mathbb Z)^2}{\varrho^2} 
  \qquad \mbox{in} \qquad 
  L^2(\R^2 \setminus \{0\}) 
  \,.
\end{equation}
The case of integer flux quanta,
i.e.\ $\alpha \in \mathbb Z$, must be excluded,
because $-\Delta_{A_\alpha}$ with any such~$\alpha$
is unitarily equivalent to $-\Delta$, which is critical. 

Our next result is the following generalization
of~\eqref{Poincare} to the Heisenberg group.
In this case, it is natural to work in cylindrical coordinates 
$(r,\varphi,z) \in (0,\infty) \times \mathbb{S}^1 \times \R$.
\begin{theorem}\label{thm:lw2-heis}
Let $A$ be an Aharonov--Bohm potential on 
$\mathbb {H}^1\setminus\mathcal Z$
with $\mathcal Z := \{(0,0,z) : z \in \R\}$. 
Then, up to gauge invariance, 
$A = A_\alpha := \alpha \, d\varphi \mod \omega$
for some $\alpha\in \mathbb R$ and the inequality  
\begin{equation}\label{eq:Hardy.Delta}
  -\Delta_{A_\alpha} \ge \frac{d(\alpha, \mathbb Z)^2}{r^2} 
  \qquad \mbox{in} \qquad 
  L^2(\mathbb {H}^1 \setminus \mathcal Z) 
\end{equation}
holds in the sense of quadratic forms,
where $r(x,y,z):=\sqrt{x^2+y^2}$. 
Moreover, the inequality is optimal in the sense that 
no positive function can be added to
the right-hand side of~\eqref{eq:Hardy.Delta}.
\end{theorem}

This is indeed a non-trivial magnetic improvement because
(even if $-\Delta$ in $L^2(\mathbb {H}^1)$ is subcritical, 
see~\eqref{eq:GL})
there exists no positive number~$c$ 
such that $-\Delta \geq c/r^2$ in $L^2(\mathbb {H}^1 \setminus \mathcal Z)$,
see \cite{Ruzhansky2017}. On top of this, Theorem~\ref{thm:lw2-heis} will be instrumental to the proof of Theorem~\ref{thm:impro}, where we show that the critical operator $-\Delta -r^2/\rho^4$ becomes subcritical after adding a magnetic field to the Laplacian.
 
The proof of Theorem~\ref{thm:lw2-heis}, presented in Section~\ref{ss:AB}, requires a careful analysis of the commutation relations 
between the Aharonov--Bohm magnetic potential and both horizontal vector fields. This is in contrast with 
the relatively simple proof of~\eqref{Hardy.LW} 
based on the polar decomposition of the Euclidean gradient 
into a ``radial" and an ``angular" direction,
which does not exist in the Heisenberg group 
(see \cite{FranceschiPrandi20}, and Remark~\ref{rmk:non-orthogonal} below).

In the Euclidean case, 
an analogous decomposition in hypercylindrical coordinates
$(r,\varphi,z) \in (0,\infty) \times \mathbb{S}^1 \times \R^{n-2}$
allows for even stronger improvements. 
Namely, in \cite{FKLV20}, 
the authors consider Aharonov--Bohm potentials 
$A_\alpha := \alpha \, d\varphi$
in $\mathbb{R}^n\setminus\{r=0\}$, 
where $n\geq 2$ and 
$r(x) := \sqrt{x_1^2+x_{2}^2}$
is the distance to the subspace $\{r=0\}$ of dimension $n-2$, 
and prove the following inequality for the Euclidean magnetic Laplacian:
\begin{equation}\label{Hardy.LW.general}
    -\Delta_{A_\alpha} - \frac{c_n}{\varrho^2} 
    \ge  \frac{d(\alpha,\mathbb Z)^2}{r^2} 
    \qquad \mbox{in} \qquad 
    L^2(\mathbb{R}^n\setminus\{r=0\}) 
    \,,
\end{equation}
where $\varrho(x) := \sqrt{x_1^2+\dots+x_{n}^2}$
is the Euclidean distance as in~\eqref{Hardy}.
That is, the Aharonov--Bohm potential improves 
the classical Hardy inequality~\eqref{Hardy} 
with a weight singular at the origin
by a term that is singular on the subspace $\{r=0\}$. 
For $n=2$, \eqref{Hardy.LW.general} reduces to~\eqref{Hardy.LW}.

Motivated by this fact, in Section~\ref{ss:follandstein}, 
we study how Aharonov--Bohm potentials 
of Theorem~\ref{thm:lw2-heis} interact with 
the Hardy-type inequality~\eqref{eq:GL} due to Garofalo and Lanconelli,
the latter being classical in the case of the Heisenberg group. 
The main result of this section is Proposition~\ref{prop:improvement-GL} where we show that the following improvement holds
\emph{under suitable symmetry assumptions}:
\begin{equation}\label{eq:impro-rot-intro}
  -\Delta_{A_\alpha} -\frac{r^2}{\rho^4}
  \ge  d(\alpha,\bZ)^2 \ \frac{1-|\nabla\rho|^4}{r^2}
  \qquad \mbox{in} \qquad 
  L^2(\mathbb{H}^1 \setminus \mathcal Z) 
  \,.
\end{equation}
%
Here, $\rho$ is the Koranyi norm, and $\nH\rho$ its horizontal gradient.
In particular, the above inequality holds for functions that are symmetric with respect to rotations around the $z$-axis, 
or with respect to the reflection $(x,y,z)\mapsto(x,y,-z)$. We stress that these symmetry assumptions frequently appear in functional inequality concerning the Heisenberg group, see, e.g., \cite{montiRearrangements2014}.

The crucial observation allowing to derive 
the improvement~\eqref{eq:impro-rot-intro} 
is the connection detailed in Lemma~\ref{l:quadform} below
between the Aharonov--Bohm magnetic Laplacian $\Delta_{A_\alpha}$ 
and the Folland--Stein operator~\cite{follandEstimates1974}
\begin{equation}\label{Folland}
    \cL_\alpha := -\Delta - i\alpha \partial_z
    \qquad \mbox{in} \qquad 
    L^2(\mathbb{H}^1 \setminus \mathcal Z) 
    \,.
\end{equation}
This allows us to deduce~\eqref{eq:impro-rot-intro} 
from the following Hardy-type inequality for $\cL_\alpha$, 
proved in Section~\ref{ss:follandstein}, 
which is interesting in its own right.
\begin{theorem}\label{thm:GL-FS}
Let $\alpha\in (-1,1)$. The inequality  
\begin{equation}
    \cL_\alpha \ge (1-\alpha^2) \, \frac{r^2}{\rho^4}
    \qquad \mbox{in} \qquad 
    L^2(\mathbb{H}^1 \setminus \mathcal Z) 
\end{equation}
holds in the sense of quadratic forms.
Moreover, the inequality is optimal in the sense that 
no positive function can be added to
the right-hand side of~\eqref{thm:GL-FS}.
\end{theorem}

\subsection{Mild magnetic fields}
We now turn our attention to physically more relevant magnetic fields, 
which lie in between the extreme situations of
uniform and Aharonov--Bohm fields.
We call them \emph{mild} for they are at the same time
\emph{regular}, 
in the sense that they are realized by smooth magnetic potentials 
(contrary to the Aharonov--Bohm potentials),
and \emph{local} in the sense that they vanish at infinity 
(contrary to the uniform fields).
Then $\sigma(-\Delta_A) = \sigma(-\Delta)$
and the quantification of the magnetic effects is more subtle.

In the Euclidean case, it is known that
$-\Delta_A-c_n/\varrho^2$ is subcritical in $L^2(\R^n)$ 
if $n \geq 2$ and $B$~is not identically equal to zero
(recall that $-\Delta-c_n/\varrho^2$ is critical
because of the optimality of~\eqref{Hardy}).
In the general setting, this was first observed 
by Weidl in~\cite{Weidl_1999},
who established the Hardy-type inequality 
$-\Delta_A-c_n/\varrho^2 \geq c(n,A,\Omega) \, \chi_\Omega$
with any compact subset $\Omega \subset \R^n$
and $c(n,A,\Omega)$ being a positive constant
depending on $n \geq 2$, $A \not= 0$ and~$\Omega$,
where $\chi_\Omega$ denoted the indicator function of~$\Omega$. 
The compactly supported Hardy weight on the right-hand side
of this inequality can replaced by a positive one
\cite[Thm.~1.1]{Cazacu2016}:
\begin{equation}\label{Hardy.Cazacu}
  -\Delta_A-\frac{c_n}{\varrho^2} 
  \geq \frac{c(n,B)}{1+r^2 \log^2 r}
  \qquad \mbox{in} \qquad 
  L^2(\R^n) 
  \,,
\end{equation}
valid for every smooth~$A$ such that $dA = B$,
where~$c(n,B)$ is a positive constant depending on
$n \geq 2$ and $B \not= 0$.

Under extra hypotheses, it is next possible to remove
the logarithm from the right-hand side of~\eqref{Hardy.Cazacu}
(see \cite[Thm.~3.2]{Cazacu2016} based on ideas of~\cite{LW99}).
In particular, this is the case if $n=2$, 
$B(x,y)=b(x,y) \, dx\wedge dy$ with a smooth function $b:\R^2 \to \R$
and the total \emph{magnetic flux}
\begin{equation}\label{flux}
  \Phi_B := \frac{1}{2\pi}\int_{\mathbb R^2} b(x,y)\,dxdy
\end{equation}
is not an integer.
A key observation is that, by Stokes theorem, the vector potential of a compactly supported magnetic field~$B$ can be chosen as 
the Aharonov--Bohm potential $\Phi_B\,d\varphi$ outside 
a compact neighborhood of the origin.

In the present setting of the Heisenberg group,
we are primarily concerned with improving the   
Hardy--Garofalo--Lanconelli inequality~\eqref{eq:GL} 
due to the presence of \emph{any} magnetic field.
\begin{theorem}\label{thm:impro}
Let $\Omega\subset \bH^1$ be a bounded open set 
with Lipschitz boundary, and $A$ be a vector potential of either  one of the following types:
\begin{enumerate}
    \item[(i)] a smooth vector potential on $\bH^1$ whose associated magnetic field $B=dA$ is such that $B\not\equiv 0$ on $\Omega$;
    \item[(ii)] an Aharonov--Bohm potential on $\bH^1\setminus\mathcal Z$ with non-integer flux (i.e., up to gauge invariance, $A = \alpha d\varphi\mod \omega$ with $\alpha \in \R\setminus\mathbb{Z}$).
\end{enumerate}
Then, there exists a positive constant $c(A, \Omega)$ dependent on~$A$ and~$\Omega$ such that
\begin{equation}\label{Hardy.main}
  -\Delta_A-\frac{r^2}{\rho^4}
  \geq c(A,\Omega) \, \chi_{\Omega}
  \qquad \mbox{in} \qquad 
  L^2(\mathbb{H}^1) 
  \,.
\end{equation}
Moreover, if $A$ is of type (i), $c(A,\Omega)$ depends only on the associated magnetic field $B$.
\end{theorem}

This result is reminiscent of that of Weidl~\cite{Weidl_1999}
in the Euclidean case mentioned above.
We prove it in Section~\ref{ss:impro} by showing that the spectrum of the shifted operator $-\Delta_A-r^2/\rho^4$ with Neumann boundary conditions on the bounded set~$\Omega$ is purely discrete and bounded away from~$0$ whenever the vector potential~$A$ satisfies the assumptions of Theorem~\ref{thm:impro}. We stress that the  proof that we present in the Aharonov--Bohm case relies on the validity of the improved Hardy inequality from the center presented in  Theorem~\ref{thm:lw2-heis}.
We leave as an open problem whether the compactly supported 
Hardy weight on the right-hand side of~\eqref{Hardy.main}
can be replaced by a positive one in the spirit of~\eqref{Hardy.Cazacu}.

Finally, we present Hardy-type inequalities, which do not necessarily
improve~\eqref{eq:GL}, but provide \emph{positive} Hardy weights  
under magnetic flux conditions.
We restrict to magnetic fields of the form
\begin{equation}\label{field.class}
  B(x,y,z)=b_1(x,y,z) \, dx\wedge d\omega + b_2(x,y,z) \, dy\wedge\omega
\end{equation}
and assume that its support is contained in a cylinder 
$\{x^2 + y^2 \le r_0\}$ for some positive~$r_0$
in the sense that it is the case of the functions 
$b_1,b_2:\R^3\to\R$.  
Notice, in particular, that $B$ could be unbounded 
with respect to the variable~$z$. 
It is useful to remark that, due to its closedness, 
$B$~is uniquely determined by its \emph{primitive}, 
that we define as
\begin{equation}\label{primitive}
    \b(r\cos\varphi, r\sin\varphi,z) 
    := - \int_r^{+\infty} b_1(t\cos\varphi, t\sin\varphi, z)\,dt 
    \,, 
\end{equation}
where $r\ge 0$, $\varphi\in \mathbb S^1$ and $z\in\mathbb R$.
\begin{theorem}\label{thm:lw1-heis}
Let~$B$ be a magnetic field on $\mathbb H^1$ 
of the form~\eqref{field.class}
and assume that its support is contained in the cylinder 
$\{x^2+y^2\le r_0^2\}$ for some $r_0>0$. 
Then, the quantity
        \begin{equation}
            \flux{B} := \frac{1}{2\pi}\int_{\mathbb R^2} \b(x,y,z)\,dxdy,
        \end{equation}
        is independent of $z\in \mathbb R$. Moreover, if $\flux{B}\notin \mathbb Z$, there exists a positive constant $c(B)$ 
        dependent on~$B$
        such that
\begin{equation}\label{eq:lw-boh}
  -\Delta_A \ge \frac{c(B)}{1+r^2} 
  \qquad \mbox{in} \qquad 
  L^2(\mathbb{H}^1) 
  \,.
\end{equation}
\end{theorem}

Let us observe that if $B$ is compactly supported then $\flux{B}=0$. In this case, it can actually be shown that the associated magnetic potential can be chosen to be compactly supported, which implies that no improvement of $-\Delta_A\ge 0$ as above is possible, 
see Remark~\ref{rmk:compact-supp} below.

\begin{remark}
In Theorem~\ref{thm:hardy.logarithm} below we prove a slightly stronger result than Theorem~\ref{thm:lw1-heis}. Indeed, we are able to replace 
the right-hand side of~\eqref{eq:lw-boh} with a function that behaves as $(r\log r)^{-2}$ as $r\downarrow 0$. We stress that the same technique can be applied also in the Euclidean case, 
yielding an improvement of~\eqref{Hardy.Cazacu}, 
that was not known to the best of our knowledge,
see Remark~\ref{rmk:euclidean.improvement}.
\end{remark}

The Euclidean magnetic Hardy-type inequalities~\eqref{Hardy.Cazacu} 
were fundamental ingredients for the study of 
the large-time behavior of the magnetic heat semigroup 
in \cite{Krejcirik2013, Cazacu2016}.
In particular, it was shown that although compactly supported magnetic fields on the Euclidean plane do not shift the spectrum, 
if they have non-integer flux, 
they improve the decay of the $L^2$-norm of the solutions of 
the heat equation with initial data living 
in some appropriate weighted space. 
An interesting 
research direction for a future work
is the application of Theorem~\ref{thm:impro}
 to show an analogous improved decay rate of the magnetic heat semigroup
in the Heisenberg setting.

\section{The Heisenberg group}
%

\subsection{The basic structure}
The Heisenberg group $\bH^1$ is $\bR^3$ 
endowed with the non-commutative group law 
\begin{equation*}
	(x,y,z)\ast (x',y',z') 
	= \left(x+x',y+y',z+z'+\frac{xy'-x'y}{2}{}\right)
	, 
\end{equation*}
where ${(x,y,z),\ (x',y',z')\in\R^3}$.
A basis for the Lie algebra of left-invariant vector fields 
is given by $X,Y$ defined in~\eqref{fields} 
together with $Z := \partial_z$.
The associated sub-Riemannian structure is given by the distribution $\mathcal D := \operatorname{span}\{X,Y\}$ endowed with the scalar product making $X$ and $Y$ orthonormal, and denoted by a dot, \emph{i.e.,}  $(a_1X+a_2Y)\cdot (b_1X+b_2Y)=a_1b_1+a_2b_2$ for $a_i,b_i\in\bC$. This structure is step $2$ since $[X,Y]=Z$, so that {$\operatorname{span}\{X,Y,[X,Y]\}|_q=T_q\bH^1$} for all $q\in\bH^1$. Moreover it is nilpotent since $[X,Z]=[Y,Z]=0$. 

What is relevant for the following is that 
the Heisenberg group structure is \emph{contact}. 
{That is, the one form $\omega\in \Omega^1(\bH)$ given by
 \begin{equation*}
  \omega(x,y,z) := dz - \frac{1}{2} (x\,dy-y\,dx)
 \end{equation*}
satisfies $\omega\wedge d\omega\neq 0$ and $\ker \omega = \operatorname{span}\{X,Y\}$.} The Reeb vector field is $Z$ (i.e., $Z\in \ker d\omega$ and $\omega(Z) = 1$), and the dual basis of the cotangent bundle {$T^*\bH^1$} associated with $\{X,Y,Z\}$ is $\{dx,dy,\omega\}$.

\subsection{The Laplacian}
The sub-Riemannian Hamiltonian $h\in C^\infty(T^*\bH^1)$ given by the Heisenberg structure is
 \begin{equation}\label{eq:ham-heis}
  h(p,q) = \frac12 \left(\langle p, X(q)\rangle^2 + \langle p, Y(q)\rangle^2 \right) = \frac12 (p_x^2+p_y^2),
 \end{equation}
 where $(p,q)\in T^*\bH^1$, 
 $\langle\cdot,\cdot\rangle$ denotes the duality between covectors 
 and vectors, 
 and $p=p_xdx+p_ydy+p_\omega\omega$. See, e.g., \cite{agrachev_barilari_boscain_2019}.
 Let us denote by 
 \begin{equation}\label{seminorm}
   |\cdot| := \sqrt{2h(\cdot)}=\sqrt{p_x^2+p_y^2}
 \end{equation} 
 the induced seminorm on $T^*\bH^1$.
Then, the sub-Riemannian Dirichlet energy $Q$ is the closed form on $L^2(\bH^1)$ with core $C^\infty_c(\bH^1)$, defined by 
\begin{equation*}
	Q(u) := \int_{\bH^1} |d u(q)|^2\,dq, \qquad \forall u\in C^\infty_c(\bH^1).
\end{equation*}
Here $d$ denotes the exterior differential, and $dq$ is the usual Lebesgue measure on $\mathbb R^3$. The associated diffusion operator is the Heisenberg sub-Laplacian $-\Delta := X^*X + Y^*Y$, which in coordinates $(x,y,z)\in \bH^1$ reads
\begin{equation}\label{eq:lapl}
  \Delta  = \partial_x^2 + \partial_y^2 + \frac{x^2 + y^2}4 \partial_z^2 + \partial_z \frac{x\partial_y-y\partial_x}2.
\end{equation}
This is a self-adjoint operator in $L^2(\bH^1)$.

The form domain of $Q$ is the horizontal Sobolev space $W^1(\bH^1)$, while $W^2(\bH^1)$ is the domain of the sub-Laplacian. 
These spaces coincide with the set of $L^2(\bH^1)$ functions such that the following respective norms, computed in the sense of distributions, are finite (see, {e.g.}, \cite{Folland73,HK2000}):
\begin{equation*}
	\|u\|_{W^1(\bH^1)} 
	:= \sqrt{\|u\|_{L^2(\bH^1)}^2+ Q(u)} \,, 
	\qquad
	\|u\|_{W^2(\bH^1)} 
	:= \sqrt{\|u\|_{W^1(\bH^1)}^2 + \|\Delta u\|_{L^2(\bH^1)}^2} \,.
\end{equation*}

\subsection{Cylindrical coordinates}\label{sec:cylindrical}

Let $\cZ := \{x=y=0\}\subset \bH^1$ be the center of the group as above, 
and consider cylindrical coordinates $(r,\varphi,z)\in\R_+\times\bS^1\times\R$
with $\R_+ := (0,\infty)$, 
so that for any $(x,y,z)\in\bH^1\setminus\mathcal Z$, we write $(x,y,z)=(r\cos\varphi,r\sin\varphi,z)$. In these coordinates, up to an orthogonal transformation in the $(x,y)$ coordinates, the basis $\{dx,dy,\omega\}$ of $T^*\bH^1$ is transformed in $\{dr, rd\varphi,\omega\}$, where
\begin{equation}
	\omega = dz -\frac{r^2}2 \, d\varphi.
\end{equation}
The corresponding dual basis for $T\bH^1$ is then
\begin{equation}\label{eq:RPhi}
	R := \partial_r, \qquad \Phi := \frac1r \partial_\varphi+\frac r2\partial_z, 
	\qquad \text{and}\qquad
	Z := \partial_z.
\end{equation}
That is, $\{R,\Phi\}$ is a global orthonormal frame for $\mathcal D$.
In these coordinates, the Laplacian (with an abuse of notation
denoted by the same symbol~$\Delta$) acts on 
the Hilbert space
$L^2(\bR_+\times\bS^1\times \bR,\, r\,drd\varphi dz)$ as
\begin{equation}\label{eq:sublap-cyl}
	\Delta = -R^*R-\Phi^*\Phi = \partial_r^2 +\frac1r \partial_r+\frac1{r^2}\left(\partial_\varphi + \frac{r^2}2 \partial_z\right)^2.
\end{equation}

\section{Magnetic fields in the Heisenberg group}

To motivate the definition of magnetic fields in the case of the Heisenberg group, we start by recalling some facts about magnetic fields in Riemannian geometry.

\subsection{Riemannian magnetic fields}

A magnetic field $B$ on a Riemannian manifold $(M,g)$ is a closed smooth $2$-form.
Locally, it is always possible to find a $1$-form $A\in \Omega^1(M)$ such that $dA=B$, known as vector potential for $B$, thanks to the identification of $1$-forms and vector fields induced by the Riemannian metric. Recalling that the Riemannian Hamiltonian $h\in C^\infty(T^*M)$ is obtained by duality with the metric $g$, the action of a magnetic field can then be interpreted as a change in the Hamiltonian function $h\mapsto h_A$, which reads
\begin{equation}\label{eq:hA}
	h_A(p,q) = h(p+A(q),q).
\end{equation}
 
The corresponding magnetic Laplacian $\Delta_A$ is obtained by an appropriate quantization of $h_A$, which yields the associated form
(recall~\eqref{seminorm})
\begin{equation}
	Q_A(u) := \int_M |(d + iA) u|^2\,dq,
	\qquad \forall u \in C^\infty_c(M).
\end{equation}
Note that, by definition, $|du| = |\nabla u|:=\sqrt{g(\nabla u, \nabla u)}$, where $\nabla$ is the Riemannian gradient.
For this reason, the above expression is sometimes written by replacing~$d$ with~$\nabla$,  implicitly identifying the $1$-form $A$ with the associated vector field. 

\subsubsection{Gauge invariance}
It is clear that if a globally defined vector potential $A$ for $B$ exists, it is not unique, as $A+\varpi$ is also a vector potential for $B$ as soon as 
$\varpi\in \Omega^1(M)$ is closed.
It can be shown that classically this does not pose any problem, as the resulting magnetic trajectories (i.e., projections on $M$ of trajectories on $T^*M$ of integral curves of the Hamiltonian vector field induced by $h_A$) depend only on $B$. 
 
This \emph{gauge invariance} is only partially true from the quantum perspective. Indeed, if $\varpi$ is exact, i.e., if $\varpi = df$ for some $f\in \Omega^0(M)=C^\infty(M)$, then $\Delta_A$ is unitarily equivalent to $\Delta_{A+df}$ via the \emph{gauge transformation} $u\mapsto \exp(i\gamma f)u$ for any $\gamma\neq 0$.
In this case, we say that $A$ and $A+\varpi$ are \emph{gauge equivalent}.
However, if the de Rham cohomology group $H^1_{\text{dR}}(M)$ is non-trivial, 
there could exist non-exact but closed forms $\varpi$, such that the quantum dynamics induced by $A$ and $A+\varpi$ are different, see e.g.\ \cite[Section~7.2]{Helffer88}. This was first observed in the case of the $2$-dimensional Euclidean magnetic Laplacian on the punctured plane for potentials of the form $\alpha d\varphi$
in the polar coordinates $(r,\varphi)$
with $\alpha\notin \mathbb Z$ 
and is known as the Aharonov--Bohm effect \cite{aharonov1959significance}.

\subsection{Horizontal magnetic fields}\label{sec:horizontal}

Since the notions of differential $k$-form only depends on the differential structure, the same definitions as above can be adapted to the Heisenberg group $\bH^1$. In this case, however, due to the degeneracy of the sub-Riemannian Hamiltonian defined in \eqref{eq:ham-heis}, we have additional simplifications:
\begin{enumerate}
	\item Writing the vector potential as $A=A_xdx+A_ydy+A_\omega \omega\in \Omega^1(\bH^1)$, the expression for $h_A$ given in \eqref{eq:hA} is independent of $A_\omega$. This is due to the fact that the sub-Riemannian Hamiltonian \eqref{eq:ham-heis} is degenerate along $\omega$. As a consequence, the natural class where to look for vector potentials defined on an open set $U\subset \bH^1$ is $\OmegaH^1(U):=\Omega^1(U)/\operatorname{span}\{\omega\}$.
	\item Let us write the magnetic field as $B = b_x dx\wedge \omega + b_y dy\wedge \omega + b_\omega dx\wedge dy$. Then, for the vector potential as above, we have
	\begin{equation}\label{eq:b-omega}
        b_\omega = X A_y - Y A_x - A_\omega.
	\end{equation}
    Since $A$ is determined up to $A_\omega$, we can always choose
    \begin{equation}\label{eq:Aomega}
    A_\omega = XA_y-YA_x,
    \end{equation}
    thus obtaining $b_\omega = 0$. 
    That is, without loss of generality we can consider magnetic fields in $\operatorname{span}\{dx\wedge\omega,dy\wedge\omega\}\subset \Omega^2(U)$.
\end{enumerate}

\subsubsection{A primer on the Rumin complex}
The above observations naturally lead to replace the de Rham complex of differential geometry, with the Rumin complex \cite{Rumin1994} from contact geometry.  
In the case of an open set $U\subset \bH^1$ this is the short exact sequence of spaces of smooth forms:
\begin{equation}
  0 \rightarrow \OmegaH^0(U) \xrightarrow{\dH} \OmegaH^1(U)\xrightarrow{D} \OmegaH^2(U) \xrightarrow{d} \OmegaH^3(U)\rightarrow 0.
\end{equation}
Here, 
\begin{itemize}
  \item $\OmegaH^0(U)$ coincide with the usual (compactly supported) $0$-forms $\Omega^0(U)= C^\infty_c(U)$.
  \item $\OmegaH^1(U) = \Omega^1(U)/\operatorname{span}\{\omega\}$ are the horizontal $1$-forms. These are in bijection with horizontal vector fields, where $A = A_X dx+A_Ydy\mod \omega$ is associated with $A = A_X X+A_Y Y$.
  \item $\OmegaH^2(U) = \operatorname{span}\{dx\wedge\omega,dy\wedge\omega\}\subset \Omega^2(U)$ are horizontal $2$-forms. Also these are in bijection with horizontal vector fields, where $B=b_1 dx\wedge\omega+b_2 dy\wedge\omega$ is associated with $B = b_2 X+b_1 Y$.
  \item $\OmegaH^3(U)$ coincide with the usual (compactly supported) $3$-forms, i.e.\ the volume forms $\OmegaH^3(U)= \operatorname{span}\{\omega\wedge d\omega\}$.
\end{itemize}
The coboundary maps appearing above are:
\begin{itemize}
  \item $\dH:\OmegaH^0(U) \rightarrow \OmegaH^1(U)$ is the horizontal differential. It acts on functions $f\in C^\infty_c(U)$ by $\dH f = (Xf)dx+(Yf)dy\mod \omega$.
  \item $D:\OmegaH^1(U)\rightarrow \OmegaH^2(U)$ is the main novelty of the Rumin complex. In the present case it is simply defined as $DA = d\tilde A$, where $\tilde A \in \Omega^1(U)$ is the unique $1$ form such that $\tilde A=A\mod\omega$ and $d\tilde A\in \OmegaH^2(U)$. In particular it turns out that $D$ is the following second order operator $D(A_Xdx+A_Ydy)=b_1dx\wedge\omega+b_2dy\wedge\omega$
  \begin{equation}\label{eq:DA-xy}
      b_1=X(XA_Y-YA_X)-ZA_X,\quad b_2=Y (XA_Y-YA_X)-ZA_Y.
  \end{equation}
  This follows by observing that $\tilde A=A+A_\omega\,\omega$, with $A_\omega$ defined in \eqref{eq:Aomega}.
  See below for the explicit expression of $DA$ in cylindrical coordinates.
  \item $d:\OmegaH^2(U) \rightarrow \OmegaH^3(U)$ is the restriction to $\OmegaH^2(U)\subset \Omega^2(U)$ of the standard exterior differential.
\end{itemize}

Since the maps above satisfy the usual requirement $D\circ \dH = d\circ D = 0$,  to the Rumin complex we can associate the cohomology groups $H^k_{\mathrm H}(U)$. These are the sets of closed forms in $\OmegaH^k(U)$ (i.e., the kernel of the coboundary operator on $\OmegaH^k(U)$) modulo exact forms (i.e., the image of the coboundary operator on $\OmegaH^{k-1}(U)$). With these definitions, it turns out that $H^k_{\mathrm H}(U)$ coincide with the usual de Rham cohomology groups on $\bR^3$.

We conclude this section by computing the precise formula for the operator $D$, in cylindrical coordinates $(r,\varphi,\omega)$, introduced in Section~\ref{sec:cylindrical}. Observe that in these coordinates we have  $\OmegaH^2(\bH^1\setminus\cZ)=\operatorname{span}\{dr\wedge\omega,rd\varphi\wedge\omega\}$ and that $df = (Rf)dr + (\Phi f)rd\varphi+(Zf)\omega$ for any $f\in C^\infty(U)$.

\begin{proposition}\label{prop:DA}
	Let $U\subset \bH^1$ and consider $A = \alpha_1 dr + \alpha_2 rd\varphi\mod \omega\in \OmegaH^1(U)$. Then, 
	\begin{equation}\label{eq:DA}
		DA = \left(R\gamma-Z\alpha_1\right)dr\wedge\omega + 
			\left(\Phi\gamma-Z\alpha_2\right)rd\varphi\wedge\omega,
	\end{equation}
	where we let 
	\begin{equation}\label{eq:gamma}
		\gamma = \frac1rR(r\alpha_2)-\Phi\alpha_1.
	\end{equation}
\end{proposition}

\begin{proof}
	Let $\tilde A = A + \gamma\omega$ where $\gamma\in C^\infty_c(U)$ is such that $d\tilde A \wedge \omega = 0$. Then, $d\tilde A \in \OmegaH^2(U)$ and thus $DA=d\tilde A$. Observing that $d\omega = -dr\wedge rd\varphi$, \eqref{eq:gamma} and \eqref{eq:DA} follow by direct computations.
\end{proof}


\subsubsection{Horizontal magnetic fields and Poincaré gauge}\label{sss:horizontal_magnetic}

We are finally in a position to define the meaning of magnetic field in the Heisenberg group.

\begin{definition}
	A \emph{horizontal magnetic field} on an open set $U\subset \bH^1$ is a closed $2$-form $B\in \OmegaH^2(U)$. A vector potential for $B$ is $A\in \OmegaH^1(U)$ such that $DA=B$.
\end{definition}

Thanks to the closure requirement, horizontal magnetic fields are determined by a single smooth function. 
Indeed, we have the following.

\begin{proposition}\label{prop:primitive}
    Let $B= b_1 dr \wedge \omega + b_2 rd\varphi\wedge\omega$ be a horizontal magnetic field. Then, $B$ is completely determined by $b_1$.
\end{proposition}

\begin{proof}
    We need to show that $b_2$ can be determined from $b_1$.
    The closure of $B$ implies
	\begin{equation*}
		0 = dB = \left( \frac1r R(r b_2)-\Phi b_1 \right) dr\wedge rd\varphi\wedge \omega.
	\end{equation*}
	Then, it suffices to integrate the function appearing above. Indeed, we obtain
	\begin{equation}\label{eq:b2}
		b_2(r,\varphi,z) = \frac{1}{r}\int_0^r \Phi b_1(t,\varphi,z)\,tdt . 
	\end{equation}
This concludes the proof of the proposition.	
\end{proof}

\paragraph{Magnetic sub-Laplacians}

We are particularly interested in the action of magnetic fields on the natural diffusion operator in $\bH^1$, the sub-Laplacian~$-\Delta$,
see~\eqref{eq:lapl}. 
In analogy with the Riemannian case, we pose the following.

\begin{definition} Let $U\subset \bH^1$ be an open set. 
	The magnetic sub-Laplacian associated with a vector potential $A\in \OmegaH^1(U)$ is 
	the ({non-negative}) operator {$-\Delta_A$} in $L^2(U)$ associated with the closure of the form $Q_A$, where
	\begin{equation}\label{eq:QA}
		Q_A(u) = \int_{\bH^1} |(\dH+iA)u|^2 \,dq, 
		\qquad \forall u\in C^\infty_c(U).
	\end{equation}
\end{definition}

The above is well defined since \eqref{eq:QA} is non-negative.
Observe that, as in the Riemannian case, $Q_A$ could be written by replacing $\dH u$ with the horizontal gradient $\nH u = (Xu)X+(Yu)Y$, identifying the horizontal $1$-form $A$ with the associated horizontal vector field, and computing the norm with respect to the dual norm on the tangent bundle such that $|\dH u| = |\nH u|$.
With this identification, for a given vector potential $A=A_xX+A_yY$ we have $-\Delta_A=X_A^*X_A+Y_A^*Y_A$.
Here, we let $X_A:=X+iA_x$, $Y_A:=Y+iA_y$. In particular, it holds
\begin{equation}\label{eq:m-subLapl}
    -\Delta_Au=-\Delta u+|A|^2u-i\left((XA_x+YA_y)u+2A\cdot \nH u\right), \qquad u \in C^\infty(\bH^1).
\end{equation}

When $U=\mathbb H^1$, following \cite[Thm~7.22]{LiebLoss}, one can show that the form domain of~$Q_A$, denoted by $W^1_A(\bH^1)$, coincides with the set of functions in $L^2(\bH^1)$ such that $Q_A(u)$ computed in the sense of distributions is finite.
This is a consequence of the diamagnetic inequality (see \cite[Thm.~7.21]{LiebLoss} for a proof that extends to the present case):
\begin{equation}\label{eq:diamagnetic}
	Q_A(u) \ge \int_{\bH^1} |\dH (|u|)|^2\,dq, \qquad \forall u\in C^\infty_c(\bH^1).
\end{equation}





\paragraph{Poincaré gauge}
As in the Riemannian case, if $A\in \OmegaH^1(U)$ is a vector potential and $f\in C^\infty(U)$, the magnetic sub-Laplacian $\Delta_{A}$ is unitarily equivalent to $\Delta_{A+\dH f}$. 
Recall that, in this case, we say that $A$ and $A+\dH f$ are \emph{gauge equivalent}.

Since $H_{\mathrm H}^1(\bH^1)=H_{\mathrm{dR}}^1(\bR^3)=\{0\}$ it follows that magnetic sub-Laplacians on $\bH^1$ are completely determined (up to unitary transformations) by the corresponding horizontal magnetic field.  

In the sequel we will consider the following generalization of the classical Poincaré (or multipolar) gauge, that allows us to choose the most convenient gauge for the vector potential. 

\begin{proposition}\label{prop:gauge}
	Let $B=b_1 dr\wedge\omega + b_2 rd\varphi\wedge\omega$ be a horizontal magnetic field on $\bH^1$. Then, a vector potential for $B$ is given by
	\begin{equation}
		A(r,\varphi, z) = \alpha(r,\varphi,z) \,d\varphi, 
		\quad\text{where}\quad
		\alpha(r,\varphi,z) = \int_0^r \int_0^t b_1(s,\varphi,z)\,ds\,tdt.
	\end{equation}
\end{proposition}

\begin{proof}
    By Proposition~\ref{prop:DA}, since $\gamma(r,\varphi,z) = \int_0^r b_1(t,\varphi,z)dt$, we have
	\begin{equation}
		DA = 
		b_1 dr\wedge\omega + \left( \Phi R \alpha - Z(r^{-1}\alpha) \right) rd\varphi\wedge\omega.
	\end{equation}
	Since $DA$ is closed, Proposition~\ref{prop:primitive} guarantees that it coincides with $B$.
\end{proof}

\begin{proposition}\label{prop:core}
	Let $B = DA$ be a horizontal magnetic field on $\bH^1$. Then, $C^\infty_c(\bH^1\setminus\cZ)$ is a core for $Q_A$.
\end{proposition}

\begin{proof}
 	It suffices to show that for all  $u \in
        C^\infty_c(\bH^1)$
        there exists a sequence $u_n \in
        C^\infty_c(\bH^1\setminus\cZ)$ such that
        $\|u_n-u\|_{L^2}+Q_B(u_n-u)\to 0$ as $n\to \infty$.
        Let $\psi:[0,1]\to[0,1]$ be a smooth function such that $\psi=0$
in a right neighborhood of~$0$ and $\psi=1$ in a left neighborhood
of~$1$ and for any natural number $n \geq 2$, let $\eta_n:\bH^1 \to[0,1]$ be
defined by
\begin{equation}
\eta_n(q) :=
\begin{cases}
  0
  & \mbox{if}\quad \abs{\xi}<1/n^2 \,,
  \\
  \psi\big(\log_n(n^2 \abs{\xi})\big)
  & \mbox{if}\quad \abs{\xi}\in[1/n^2,1/n] \,,
  \\
  1
  & \mbox{if}\quad  \abs{\xi} > 1/n.
\end{cases}
\end{equation}
Let $u \in C^\infty_c(\bH^1)$ and for all $n\geq 2$ let
$u_n := \eta_n u \in C^\infty_c(\bH^1\setminus\cZ)$. By dominated
convergence, $u_n \to u$ in $L^2(\bH^1)$ as $n \to \infty$.
Moreover, 
we have 
\begin{equation}
  \begin{split}
    Q_A(u_n - u)
     = \, &
    \int_{\bH^1} \abs{(\eta_n -1)(\dH + i A)u + u \,
      \dH\eta_n}^2 \, dq \\
   \leq \, &
  2\int_{\bH^1} \abs{\eta_n -1}^2 \abs{(\dH + iA) u}^2 \, dq\\
  & +
  2 \frac{\norm{u}_\infty^2 \norm{\psi'}_\infty^2}{\log^2 n}
  \int_{\{\frac{1}{n^2}<\abs{\xi}<\frac{1}{n}\}\cap\supp u} \frac{1}{\abs{\xi}^2} \,
  dq = I_n + J_n.
\end{split}
\end{equation}
Since there exists $C>0$ such that
\begin{equation}
  \int_{\{\frac{1}{n^2}<\abs{\xi}<\frac{1}{n}\}\cap\supp u} \frac{1}{\abs{\xi}^2} \,
  dq
  \leq C \log n,
\end{equation}
we obtain that $J_n \to 0$ as $n\to \infty$. Moreover, by the dominated convergence theorem, also $I_n\to 0$ as $n\to \infty$, concluding the proof.
\end{proof}

\section{Magnetic Poincaré inequalities}

In this section we focus on how uniform magnetic fields interact with the bottom of the spectrum of the corresponding sub-Laplacian
and prove Theorem~\ref{thm:uniform}. 

\begin{definition}
    We say that $B$ is a \emph{uniform magnetic field} on $\mathbb{H}^1$ if there exists $b = (b_1,b_2)\in \mathbb{R}^2$ such that 
    \begin{equation}\label{field.uniform}
        B = b_1\, dx\wedge\omega + b_2\, dy\wedge\omega.
    \end{equation}
    In this case, we let $|B| := \sqrt{|b_1|^2 +|b_2|^2}$.
\end{definition}
We start by choosing a convenient gauge.

\begin{proposition}\label{prop:gauge-uniform}
    Let $(b_1,b_2)\in \mathbb{R}$ 
    and consider the uniform magnetic field~\eqref{field.uniform}.
    Then, up to a linear change of variables we have
    \begin{equation}\label{eq:B-unif}
        B = |B|\, dx\wedge \omega.
    \end{equation}
    A corresponding vector potential is
    \begin{equation}
        A = |B|\,  \frac{x^2}2  dy.
    \end{equation}
\end{proposition}

\begin{proof}
    To reduce the form of $B$ to \eqref{eq:B-unif} it suffices to consider the change of variables $(x',y')^\top= R(x,y)^\top$, where $R$ is the rotation matrix such that $(|B|,0)^\top = R(b_1,b_2)^\top$. The fact that $A = |B|\frac{x^2}{2}dy$ is a vector potential for $B$ follows from the explicit computation of $DA$ presented in \eqref{eq:DA-xy}.
\end{proof}

Before we proceed with the proof of Theorem~\ref{thm:uniform},
let us make a couple of comments on the result.

\begin{remark}
    \label{rem:montgomery}
    The non-linear growth with respect to~$|B|$ 
    in~\eqref{Poincare.Heisenberg}
    has already been identified in \cite{Montgomery1995}, for magnetic fields in the Euclidean plane vanishing on curves. Indeed, to treat this case, the author associates to these magnetic fields a non-equiregular sub-Riemannian problem (on the model of the Martinet distribution) and then applies a desingularization procedure that brings him to the Engel group (i.e., the Carnot group of step $3$ and rank $2$). His techniques translate verbatim to the present case, where we could proceed by ``lifting'' the uniform magnetic field in the Heisenberg group (the Carnot group of step $2$ and rank $2$) to the sub-Laplacian of the Engel group. This is a consequence of the relations
    \begin{equation}\label{eq:commutatori-heis-magnetic}
        -i b_1 = [[X+iA_x, Y+iA_y], X+iA_x], \qquad -ib_2 = [[X+iA_x, Y+iA_y], Y +iA_y],
    \end{equation}
holding for a magnetic field $B = b_1 dx\wedge \omega + b_2 dy\wedge\omega$ with magnetic potential $A = A_x dx + A_y dy \mod \omega$ and deduced by \eqref{eq:DA-xy}.
\end{remark}

\begin{remark} In the Euclidean plane, a better estimate can be proved. 
    Namely, for a magnetic field $B(x,y)\; dx\wedge dy=dA$ in $\R^2$, with $A=A_xdx+A_ydy$ and $B\geq 0$, there holds $-\Delta_A\ge B$ in the sense of quadratic forms, see e.g.~\cite[Thm.~3]{Montgomery1995} and references therein. This follows from the fact that in two dimensions we have
    \begin{equation}\label{eq:comm}
    -iB = [\partial_x + iA_x, \partial_y +iA_y].
    \end{equation}
    In the higher dimensional case, weaker semiclassical results  are available \cite{Helffer97-ln}. 
    As observed above, \eqref{eq:comm} is not available in the Heisenberg setting, being replaced by \eqref{eq:commutatori-heis-magnetic}.
\end{remark}

\begin{proof}[Proof of Theorem~\ref{thm:uniform}.]
    Let $A$ be the magnetic potential in the gauge introduced in Proposition~\ref{prop:gauge-uniform}. Then, the magnetic sub-Laplacian reads
    \begin{equation}
        -\Delta_A = -X^2 - \left(Y + i{|B|}\frac{x^2}{2}\right)^2.
    \end{equation}

    To simplify the computations, we perform the change of variables $z\mapsto w = z+\frac{xy}{2}$. In the new coordinates $(x,y,w)$, the vector fields generating the distribution read
    \begin{equation}
        X = \partial_x 
        \qquad\text{and}\qquad
        Y = \partial_y + x\partial_w.
    \end{equation}
    Hence, 
    \begin{equation}
        -\Delta_A = -\partial_x^2 - \left( \partial_y +x \partial_w+i |B| \frac{x^2}2 \right)^2.
    \end{equation}

    Performing a Fourier transform in the $y$ and $z$ variables, we obtain the following decomposition
    \begin{equation}
        -\Delta_A = \bigoplus_{\eta,\nu\in\mathbb{R}}L^{\eta,\nu}, 
        \qquad
        {L^{\eta,\nu} = -\partial_x^2+\left( \nu+x\eta-|B|\frac{x^2}2 \right)^2}.
    \end{equation}
    Finally, the change of variables {$t = x-\frac{\eta}{|B|}$}  allows us to write
    \begin{equation}
        L^{\eta,\nu} = L^g := -\partial_t^2 + \left(\frac{|B|}{2}t^2+g\right)^2, \qquad \
        g=-\frac{\eta^2}{2|B|}-\nu.
    \end{equation}
    The statement then follows by Lemma~\ref{lem:quartic} below.
\end{proof}

The following result is well-known (see, e.g., \cite[Sec.~5]{Montgomery1995}). We present a proof for completeness.

\begin{lem}
    \label{lem:quartic}
    For $b>0$ and $g\in \mathbb{R}$ let us consider 
    the maximal realization of
    the operator  
    \begin{equation}
        L^g_b := -\partial_t^2 {+\left(\frac{b}{2}t^2 + g\right)^2}
        \qquad\mbox{in}\qquad
        L^2(\mathbb{R})
        \,.
    \end{equation}   
    Then its spectrum satisfies
    \begin{equation}
        \min_{g\in \mathbb{R}} \inf \spec(L_b^g) = c b^{2/3}, \qquad c=\min_{g\in \mathbb{R}} \inf\spec(L_1^g) >0.
    \end{equation}
\end{lem}

\begin{proof}
{First of all,
we recall (see, e.g., \cite{Evans-Zett_1978})
that $L^g_b$ is essentially self-adjoint on $C^\infty_c(\mathbb R)$
and that its closure coincides with the maximal realisation.
Moreover, one has the useful separation property
$\dom(L^g_b)=\{u \in W^{2,2}(\R)\,|\, u \in L^2(\R,t^4 \, dt)\}$.
}

   We start by performing the change of variables $s=b^{1/3}t$, obtaining
   \begin{equation}
   L_b^g = -b^{2/3} \partial_s^2 +\left( b^{1/3} \frac{s^2}2+g \right)^2 = b^{2/3} L_1^{g/\sqrt[3]{b}}.
   \end{equation}
  Observe that for any $g\in\R$ the operator $L_1^{g}=-\partial_s^2+V_g$, with $V_g(s)=({s^2}/2+g)^2$, has discrete spectrum
  (see e.g.,  \cite[Thm.~XIII.67]{RS4}).
  Let $\lambda(g)$ denote its first eigenvalue. That is, $\lambda(g)=\min\spec(L_1^g)$. 
  By the min-max principle, we obtain that $\lambda(g)>0$ for any $g\in \mathbb R$.
  To complete the proof of the statement we need to show that the map $\lambda:g\in\mathbb R \mapsto\lambda(g)\in (0,+\infty)$ admits a positive minimum. 
  
  Note that $\lambda$ is continuous
  (see, e.g., \cite[VII, Thm.~3.9]{Kato}) 
  and that for any $g\geq 0$ we have $L_{1}^{g}\geq L_1^0>0$.  
  Moreover, we claim that that $\lambda(g)\to+\infty$ as $g\to-\infty$. In fact, if $g<0$, we have 
  \begin{equation}
  V_g(s)= g^2 \left(\frac{s^2}{2|g|}-1\right)^2=g^2 \left(\frac{|s|}{\sqrt{2|g|}}-1\right)^2\left(\frac{|s|}{\sqrt{2|g|}}+1\right)^2\geq g^2 \left(\frac{|s|}{\sqrt{2|g|}}-1\right)^2. 
  \end{equation}
  The claim, and thus the statement, follows since $\sqrt{|g|/2}$ is the first eigenvalue of the harmonic oscillator \begin{equation}
      -\partial_s^2+ \left(\sqrt{\frac{|g|}{2}}s-|g|\right)^2. 
  \end{equation}
This concludes the proof of the lemma.  
\end{proof}

\section{Magnetically induced Hardy inequalities from the center}


In this section we show that magnetic fields induce 
Hardy-type inequalities from the center~$\mathcal{Z}$ 
of the Heisenberg group.
More specifically, in Sections~\ref{ss:AB} and~\ref{ss:cyl}
we respectively prove Theorem~\ref{thm:lw2-heis}
for Aharonov--Bohm potentials 
and Theorem~\ref{thm:lw1-heis}
for mild cylindrically supported magnetic fields.  

\subsection{Hardy inequalities from the center for Aharonov--Bohm potentials}\label{ss:AB}
In this subsection, we are concerned with inequalities of the form 
\begin{equation}\label{eq:Hardy-center}
  -\Delta_A\ge \frac{c}{r^2}
  \qquad \mbox{in} \qquad L^2(\mathbb{H}^1 \setminus \mathcal{Z})
  \,,
\end{equation}
where $c > 0$, $r(x) := \sqrt{x^2+y^2}$
is the distance to the $z$-axis
and~$A$ is an Aharonov--Bohm potential.
Such inequalities fail in the magnetic-free case;
indeed, if $B=0$, it is known \cite{Ruzhansky2017}
that the best constant in the above inequality is $c=0$.  

Motivated by the classical Aharonov--Bohm phenomenon, we pose the following.

\begin{definition}\label{def:AB}
	An Aharonov--Bohm potential is a non-exact but closed vector potential.
\end{definition}

Observe that if $A$ is an Aharonov--Bohm potential on $U\subset \bH^1$ the corresponding magnetic field is $B=DA=0$. In this section we will show that, as in the Euclidean case, the associated magnetic sub-Laplacian may not be unitarily equivalent to the unperturbed Heisenberg sub-Laplacian.  

Clearly, for an Aharonov--Bohm potential to exist on $U\subset \bH^1$ it is necessary that $H_{\mathrm H}^1(U)\neq 0$. As a consequence, as already implicitly observed, no Aharonov--Bohm potential exist on $\bH^1$.
The simplest case with non-trivial cohomology is $U=\bH^1\setminus \cZ$, 
where $\cZ:=\{x=y=0\}$. In this case,
\begin{equation}\label{eq:cohomol}
	H_{\mathrm H}^1(\bH^1\setminus\cZ)=H_{\mathrm dR}^1(\R^3\setminus\cZ)=\bR.
\end{equation}


\begin{proposition}\label{prop:AB}
    Up to gauge invariance, any Aharonov--Bohm potential in $\bH^1\setminus\cZ$ is of the form	
	\begin{equation}\label{eq:AB}
		A_\alpha={\alpha} \, d\varphi, 
		\qquad 
		\alpha\in \bR\setminus\mathbb Z.
	\end{equation}
	Moreover, ${A_\alpha}$ is gauge equivalent to $A_{\tilde \alpha}$ whenever $\alpha = n + \tilde\alpha$, with $n\in\mathbb Z$ and $\tilde \alpha\in \mathbb R$.
\end{proposition}

\begin{proof}
    Proposition~\ref{prop:DA} implies at once that $A_\alpha$ 
    is closed for any $\alpha\in\R$, 
    i.e., $DA_\alpha = 0$. 
    
    We claim that, for $\alpha,\beta\in \R$ with $\alpha\neq \beta$, the forms $A_\alpha$ and $A_\beta$ belong to different cohomology classes. 
    In view of \eqref{eq:cohomol}, this ensures that any closed form $\varpi$ is gauge equivalent to $A_\alpha$ for some $\alpha\in \R$ (i.e., $\varpi = A_\alpha + \dH f$ for $\alpha\in \mathbb R$ and $f\in C^\infty(\bH^1\setminus\cZ)$). In particular, any Aharonov--Bohm potential is gauge equivalent to $A_\alpha$ for some $\alpha\in \mathbb R$.

	To prove the claim, let $\alpha\neq \beta$ and assume by contradiction that there exists $f\in C^\infty(\bH^1\setminus\cZ)$ such that  $A_\alpha=A_\beta+\dH f$. This implies that {$Rf=0$ and} $\Phi f = ({\alpha-\beta})/r$. In particular, the first relation implies that $f$ is independent of the $r$ variable. This fact, together with the second relation, yields that $\partial_z f=0$ and $\partial_\varphi f =\alpha-\beta$. This implies that $f$ depends only on $\varphi$, which yields the desired contradiction. Indeed, by Stokes' Theorem we have
	\begin{equation}
	    0 = \int_{\mathbb S^1} \partial_\varphi f \,d\varphi = 2\pi(\alpha-\beta)\neq 0.
	\end{equation}
	
    To complete the proof, it suffices to observe that, if $\alpha = n +\tilde \alpha$ with $n\in\mathbb Z$, $A_\alpha$ is gauge equivalent to $A_{\tilde \alpha}$ via the gauge transformation $u \mapsto v:=e^{-in\varphi}u$. In particular, this implies that $A_n$ is gauge equivalent to the zero potential $A_0$, and thus that any Aharonov--Bohm potential is gauge equivalent to $A_\alpha$ for some $\alpha\in \mathbb R\setminus\mathbb Z$, i.e., as in \eqref{eq:AB}.
\end{proof}

We observe that, in the Euclidean case, the Aharonov--Bohm Laplacian admits various self-adjoint realizations \cite{AdamiTeta,dabrowski1998aharonov, pankrashkin2011spectral}. 
In the following we are only concerned
with the Friedrichs extension
of the magnetic sub-Laplacian initially defined on 
$C_c^\infty(\bH^1\setminus\cZ)$.

We are now ready to prove Theorem~\ref{thm:lw2-heis}, 
showing that \eqref{eq:Hardy-center} does indeed hold 
for the Aharonov--Bohm magnetic potentials
with an explicit constant~$c$ excluding the case
of integer flux quanta. 
The result is thus reminiscent of the celebrated 
magnetic Hardy inequality~\eqref{Hardy.LW} 
due to Laptev and Weidl~\cite{LW99} in the Euclidean case.
However, in the present situation,
the magnetic-free Heisenberg Laplacian is already subcritical.
The novelty here is that the singularity of the Hardy weight 
on the right-hand side of~\eqref{eq:Hardy-center}
is not permitted without the magnetic field.

\begin{remark}\label{rmk:non-orthogonal}
  In spite of the apparent similarity between the inequalities 
  \eqref{Hardy.LW} and \eqref{eq:Hardy.Delta},
  also their proofs are different. Indeed, in the Euclidean setting the
  thesis descends from the positivity of only the angular part of the
  magnetic Laplacian, while in the Heisenberg setting this alone is
  not sufficient.
    In detail, in the Euclidean case one has,
    for all {$\psi \in C_c^\infty(\R^2 \setminus \{0\})$},
  \begin{equation*}
    \begin{split}
    \langle -\Delta_{A_\alpha} \psi , \psi\rangle_{L^2(\R^2)}
    &=
      \int_{0}^{+\infty} \int_{\bS^1}
      \left[\abs{\partial_r \psi}^2 
      + \frac{\abs{( \partial_{\varphi}
          + i \alpha)\psi}^2}{r^2}\right]
    \, r dr d\varphi
    \\	
    &\geq \operatorname{dist}(\alpha,\bZ)^2 \int_{\R^2} \frac{\abs{\psi}^2}{\varrho^2}  \,dx,
      \end{split}
  \end{equation*}
since $\abs{\lambda} \geq \operatorname{dist}(\alpha,\bZ)$ for $\lambda \in
\sigma(-i\partial_{\varphi} +  \alpha)$
and the radial part is simply neglected. 
Proceeding analogously in the Heisenberg setting, one has, 
for all {$\psi \in C_c^\infty(\mathbb{H}^1 \setminus \mathcal{Z})$},
  \begin{equation*}
    \begin{split}
    \langle -\Delta_\alpha \psi , \psi\rangle_{L^2(\mathbb{H}^1)}
    & =
      \int_{0}^{+\infty} \int_{\bS^1 \times \R}
      \left[\abs{\partial_r \psi}^2 
      + \frac{\abs{( \partial_{\varphi}
    + \frac{r^2}{2}\partial_z + i \alpha)\psi}}{r^2}\right]
    \, r dr d\varphi dz
    \\
      & \geq \int_{\bH^1} \nu_\alpha(r)
      \frac{\abs{\psi}^2}{r^2}  \,dq,
      \end{split}
  \end{equation*}
  where for all $r>0$ we let
  \begin{equation*}
    \begin{split}
      \nu_\alpha (r)  
      :=\ & \inf_{0 \neq \theta \in W^1(\bS^1 \times \R)}
      \frac{\int_{\bS^1 \times \R} \big\vert \big( \partial_{\varphi}
        + \frac{r^2}{2}\partial_z + i \alpha\big)
        \theta(\varphi,z) \big\vert^2 \, d\varphi dz}%
      {\int_{\bS^1 \times \R} |\theta(\varphi,z)|^2 \, d\varphi dz },
      \\
      =\ &
      \inf_{0 \neq \phi \in W^1(\bS^1 \times \R)}
      \frac{\int_{\bS^1 \times \R} \big\vert \big( \partial_{\varphi}
        + \partial_\gamma + i \alpha\big)
        \phi(\varphi,\gamma) \big\vert^2  \, d\varphi d\gamma}%
    {\int_{\bS^1 \times \R} |\phi(\varphi, \gamma)|^2 \,
      d\varphi d\gamma },
  \end{split}
\end{equation*}
having performed the change of variables $z = \gamma r^2 / 2$ in
the last equality. We can not conclude 
as in the Euclidean case above, 
since $\sigma(-i\partial_{\varphi}  -i\partial_\gamma + 
\alpha) = \R$ implies that $\nu_\alpha(r) = 0$ for all $r>0$.
\end{remark}

\begin{proof}[Proof of Theorem~\ref{thm:lw2-heis}]
    Via a Fourier transform in the $\varphi$ variable, we have the decomposition
    \begin{equation}\label{eq:decomp}
        {-\Delta_{A_\alpha}} = \bigoplus_{m\in\bZ} L_m, 
    \end{equation}
    where $L_m$ denotes the operator on $L^2(\bR_+\times \bR, r dr dz)$ associated with the quadratic form {defined for $u\in C^\infty_c(\bR_+\times \bR)$ as}
    \begin{equation*}
        Q_m(u) = \int_{\bR_+\times \bR} 
        \left(|Ru|^2 + \big|\hat\Phi_{\alpha-m} u\big|^2\right)
        \, r dr dz, \qquad \hat \Phi_{\alpha-m} u := -i\frac{r}{2}\partial_z u + \frac{\alpha-m}r u.
    \end{equation*}
    Thus, to prove inequality \eqref{eq:Hardy.Delta} it suffices to show that 
    \begin{equation}\label{eq:Lm}
        {L_m} \ge \frac{\max\{-1, \min\{1,\alpha-m\}\}^2}{r^2}.
    \end{equation}
    We prove it for $m=0$. The general case easily follows.

    Let $\aa := \max\{-1,\min\{1,\alpha\}\}$, 
    fix $u\in C^\infty_c(\bR_+\times \bR)$, 
    and define $Pu := Ru + \varepsilon\hat \Phi_{\alpha-\aa}u $, where $\varepsilon = \pm 1$ will be fixed later depending on $u$. Direct computations yield
    \begin{equation*}
        0 \le  |Pu|^2 = |Ru|^2 + |\hat\Phi_{\alpha}u|^2 -\frac{\aa^2}{r^2}|u|^2 +2\Re \left( \left( \varepsilon Ru - \frac{\aa}{r}u \right)  \overline{ \hat\Phi_{\alpha-\aa} u }  \right)  
        .    
    \end{equation*}
    Integrating the above with respect to
    the measure $r drdz$, we reduce the proof of \eqref{eq:Lm} to the non-positivity of the integral of the last term above. By developing the product, and denoting the scalar product in $L^2(\bR_+\times \bR, rdrdz)$ with $\left\langle\cdot,\cdot \right\rangle$, and the associated norm by $\|\cdot\|$, this is equivalent to 
    \begin{equation}\label{eq:boh}
        \varepsilon\Re \left( \frac12\left\langle Ru, -ir\partial_z u \right\rangle +\left\langle Ru, \frac{\alpha-\aa}{r} u \right\rangle  \right) - \frac{\aa}{2}\left\langle u, -i\partial_z u \right\rangle - \aa(\alpha-\aa)\left\| \frac{u}{r} \right\|^2 \le 0.
    \end{equation}
    Here, we used the fact that $-i\partial_z$ is a symmetric operator to guarantee that $\left\langle u, -i\partial_z u \right\rangle$ is real. 
    Integrating by parts and using the fact that $\alpha$ is real, 
    we obtain
    \begin{equation}
       \Re \left\langle Ru, -ir\partial_z u \right\rangle = \left\langle u, -i\partial_z u \right\rangle
       \qquad\text{and}\qquad
        \Re \left\langle Ru, \frac{\alpha-\aa}{r} u \right\rangle = 0.
    \end{equation}
    Finally, since $\aa(\alpha-\aa)\le 0$, we have reduced \eqref{eq:boh} to
    \begin{equation}
        \frac{1}{2}(\varepsilon -\aa)\left\langle u, -i\partial_zu \right\rangle \le 0.
    \end{equation}
    Since $\aa\in [-1,1]$, one easily checks that choosing $\varepsilon := -\operatorname{sgn}\left\langle u, -i\partial_z u \right\rangle$ yields the desired inequality, proving the first part of the statement.
     
Let us now turn to an argument for the sharpness of \eqref{eq:Hardy.Delta}.
Let $\mm\in \bZ$ be such that $\operatorname{dist}(\alpha,\bZ) = |\alpha-\mm|$. By  \eqref{eq:decomp}, it suffices to construct a sequence $(\psi_n)_{n\in \bN}\subset \dom(L_{\mm})$, such that
\begin{equation}\label{eq:hardy-sharp-barm}
    \lim_{n} \frac{Q_{\mm}(\psi_n)}{\|\psi_n / r\|^2} \leq |\alpha-\mm|^2.
\end{equation}

Inspired by  \cite[proof of Corol.~VIII.6.4]{edmunds2018spectral} and
\cite{Cazacu2016},
let~$\xi:[0,1]\to[0,1]$ be a smooth function such that $\xi=0$
in a right neighborhood of~$0$ and $\xi=1$ in a left neighborhood of~$1$.
For any natural number $n \geq 2$,
we define the following smooth cut-off function $\eta_n:(0,\infty)\to[0,1]$ by
\begin{equation}
\eta_n({r}) :=
\begin{cases}
  0
  & \mbox{if}\quad {r}<1/n^2 \,,
  \\
  \xi\big(\log_n(n^2 {r})\big)
  & \mbox{if}\quad {r}\in[1/n^2,1/n] \,,
  \\
  1
  & \mbox{if}\quad {r} \in(1/n,n) \,,
  \\
  \xi\big(\log_n(n^2/{r})\big)
  & \mbox{if}\quad {r}\in[n,n^2] \,,
  \\
  0
  & \mbox{if}\quad {r} >n^2 \,.
\end{cases}
\end{equation}
The function $\eta_n$ is a convenient smooth approximation in $W^1(\bR_+,{r} d{r})$ of the constant function~$1$:
 it is immediate to see that $\eta_n \to 1$ almost everywhere
as $n \to \infty$ and that, for any $n\geq 2$, it holds
\begin{equation}\label{eq:eta.convergence}
  2 \log n
  \leq  \int_0^{+\infty} \frac{\abs{\eta_n({r})}^2}{{r}}\,d{r}
  \quad\text{and}\quad
    \int_0^{+\infty} \abs{\eta_n'({r})}^2 \, {r} d{r}
  \leq \frac{2\norm{\xi'}_\infty^2}{\log n}.
\end{equation}

For any $n\geq 2$, we now define $\chi_n : [0,+\infty) \to [0,1]$ as follows:
\begin{equation}
  \chi_n (s) :=
  \begin{cases}
  1
  & \mbox{if}\quad s < n^4 \,,
  \\
  \xi\left(2 - \frac{s}{n^4} \right)
  & \mbox{if}\quad s\in[n^4,2n^4] \,,
  \\
  0
  & \mbox{if}\quad s >2 n^4 \,.
\end{cases}
\end{equation}
Straightforward computations yield that
\begin{equation}\label{eq:chi.Hardy}
 2 n^4 \leq \int_{-\infty}^{+\infty} \abs{\chi_n(\abs{s})}^2 \, ds
  \quad \text{and}\quad
  \int_{-\infty}^{+\infty} \abs{\chi_n'(\abs{s})}^2 \, ds
  \leq \frac{2\norm{\xi'}_{\infty}^2}{n^4}.
\end{equation}
Finally, we set
\begin{equation*}
    \psi_n(r,z) = \eta_n(r) \chi_n(|z|), \qquad \text{ for all }(r,z) \in \bR_+\times \bR.
\end{equation*}

We estimate the Hardy--Rayleigh quotient 
${Q_{\mm}(\psi_n)}/{\|\psi_n / r\|^2}$. 
We have
\begin{equation}\label{eq:Hardy.1}
    \begin{split}
        Q_{\mm}(\psi_n) 
        &= \int_\bR \int_0^{\infty} \left[ \abs{\partial_{r} \psi_n}^2 + \left|\left(-i \frac{r}2 \partial_z + \frac{ \alpha-\mm}r\right) \psi_n\right|^2 \right] \,  {r} d{r} dz\\
        &= \int_\bR \int_0^{\infty} \left[ \abs{\eta'_n(r)\chi_n(|z|)}^2 + \left|\left(-i \frac{r}2 \partial_z + \frac{\alpha-\mm}r\right) \eta_n(r)\chi_n(|z|)\right|^2 \right] \,  {r} d{r} dz\\
        & = I_1+I_2+I_3,
    \end{split}
\end{equation}
where we have defined 
\begin{align*}
  I_1 &:=  \int_\bR  \int_0^{\infty} \abs{\eta_n'({{r}})\chi_n(\abs{z})}^2    \,  {r} d{r} dz,\\
  I_2 &:=  \int_\bR  \int_0^{\infty} \frac{\abs{(\alpha-\mm)\eta_n({{r}})\chi_n(\abs{z})}^2}{{r}^2} \,  {r} d{r} dz,\\
  I_3 &:=  \int_\bR  \int_0^{\infty}
          \abs{\eta_n({{r}})}^2 \left|\frac{r}{2}\partial_z\chi_n(\abs{z})\right|^2\,
          r d{r} dz.
\end{align*}
and in the last equality we used the fact that $\psi_n$ is real.
 
Thanks to \eqref{eq:eta.convergence}, and the explicit expression of $\psi_n$, we have
\begin{equation}
\frac{I_1}{\|\psi_n /r\|^2}
  = \frac{\int_0^{\infty}
    \abs{\eta_n'(r)}^2
    \,  {r} d{r} }
    {\int_0^{\infty}
    {\abs{\eta_n({{r}})}^2}{{r^{-1}}}
    \,  d{r}}
    \le \frac{\|\xi'\|_\infty^2}{\log^2 n} \xrightarrow[n\to\infty]{} 0.
\end{equation}
On the other hand, a direct computation immediately shows that
\begin{equation}\label{eq:Hardy.3}
\frac{I_2}{\|\psi_n/r\|^2}  \xrightarrow[n\to\infty]{} |\alpha-\mm|^2 .
\end{equation}
Finally, since on the support of $\eta_n$ it holds ${r} \leq n^2$,  thanks to \eqref{eq:eta.convergence} and \eqref{eq:chi.Hardy} we estimate
\begin{equation}\label{eq:Hardy.5}
  \frac{I_3}{\|\psi_n/r\|^2}
  \leq
  \frac{\|\xi'\|_\infty^2}{32 \log n} \xrightarrow[n\to\infty]{} 0.
\end{equation}
The desired estimate thus \eqref{eq:hardy-sharp-barm} follows by  putting together \eqref{eq:Hardy.1}--\eqref{eq:Hardy.5}. 
\end{proof}

\subsection{Hardy-like inequalities for cylindrically supported magnetic fields}\label{ss:cyl}
In this subsection, we are concerned with inequalities of the form 
\begin{equation}\label{eq:Hardy-center.mild}
  -\Delta_A\ge \frac{c}{1+r^2}
  \qquad \mbox{in} \qquad L^2(\mathbb{H}^1)
  \,,
\end{equation}
where $c > 0$, $r(x) := \sqrt{x^2+y^2}$
is the distance to the $z$-axis
and~$A$ is a vector potential
for any mild magnetic field.
   
    Laptev and Weidl showed in~\cite{LW99} 
    (see also \cite[Thm.~3.2]{Cazacu2016})
    that~\eqref{eq:Hardy-center.mild} holds in the two-dimensional 
    Euclidean case with $c=c(B)>0$ as soon as the magnetic field vanishes at infinity and the total magnetic flux~\eqref{flux} is not an integer. 
	
In the following we establish Theorem~\ref{thm:lw1-heis},
which contains an estimate of type~\eqref{eq:Hardy-center.mild}.
In fact, we prove it as a consequence of a stronger estimate
contained in~Theorem~\ref{thm:hardy.logarithm},
where the Hardy weight is allowed to be singular on~$\mathcal{Z}$.
What is more, our strategy allows us to improve~\eqref{eq:Hardy-center.mild}
also in the Euclidean setting, 
see Remark \ref{rmk:euclidean.improvement} below.

Recall that the primitive $\b$ of an horizontal magnetic field $B = b_1 dr\wedge \omega + b_2 rd\varphi\wedge \omega$, such that $\supp B\subset \{r\leq r_0\}$ for some $r_0>0$, is defined in \eqref{primitive}. In cylindrical coordinates, this reads
\begin{equation}\label{eq:prim}
        {\b(r,\varphi, z) 
        =- \int_r^{+\infty} b_1(t, \varphi, z)\, dt}.
\end{equation}

\begin{remark}
Since $R\b = b_1$, by Proposition~\ref{prop:primitive} a horizontal magnetic field supported on a cylinder is completely determined by its primitive.
\end{remark}

In the next proposition we show the relation between the support of a horizontal magnetic field and of its primitive. 
\begin{proposition}\label{prop:suppcomp}
  The horizontal magnetic field $B = b_1 dr\wedge \omega + b_2 rd\varphi\wedge \omega$ is supported in the cylinder $\{r\leq r_0\}$ if and only if its primitive $\b$ is supported in the cylinder $\{r\leq r_0\}$ and it satisfies
  \begin{equation}\label{eq:suppcomp}
  \int_0^{r_0}Z\b(t,\varphi,z)\;tdt=0,\quad \forall \varphi\in\bS^1, z\in\bR.
  \end{equation}
  \end{proposition}
  \begin{proof}
  By \eqref{eq:prim}, $\supp b_1\subset\{r\leq r_0\}$ if and only if $\supp \b\subset\{r\leq r_0\}$. 
 We claim that, for $r>0,\varphi\in\bS^1,z\in\R$, we have 
 \begin{equation}\label{eq:b2supp}
     b_2(r,\varphi,z)=\Phi\b(r,\varphi,z)-\frac{1}{r}\int_0^rZ\b(t,\varphi,z)\;tdt.
 \end{equation}
Combining the last two facts, the statement follows since $\supp B\subset\{r\leq r_0\}$ if and only if
  \begin{equation*}
  \supp\b(r,\varphi,z)\subset\{r\leq r_0\}
  \quad\text{and}\quad\int_0^{r_0} Z\b(t,\varphi,z)\,tdt=0.
  \end{equation*}  
  
  We prove \eqref{eq:b2supp}. Since $R\b=b_1$, formula \eqref{eq:b2} reads,
  for any $r>0,\varphi\in\bS^1,z\in\R$,
  \begin{equation}\label{eq:b21}
    rb_2(r,\varphi,z) = \int_0^r \Phi R\b(t,\varphi,z)\,tdt.
  \end{equation}
  Moreover, there holds 
  \begin{equation*}
      [\Phi,R]=-\left(-\frac{1}{r^2}\partial_\varphi+\frac{1}{2}\partial_z\right)=\frac{1}{r}\Phi-Z,
  \end{equation*}
  so that \eqref{eq:b21} reads, for $r>0,\varphi\in\bS^1,z\in\R$,
  \begin{equation}\label{eq:b22}
    rb_2(r,\varphi,z) = \int_0^r R\Phi\b(t,\varphi,z)\,tdt-\int_0^rZ\b(t,\varphi,z)\;tdt+\int_0^r\Phi\b(t,\varphi,z)\;dt.  
  \end{equation}
  Integrating by parts the first term of the previous equation yields the claim.
\end{proof}

\begin{rmk}\label{rmk:exa-notbdd}
Let us consider the magnetic field $B= b_1 dr\wedge \omega + b_2 rd\varphi\wedge \omega$ with $b_1(r,\varphi,z)=\chi(r)\beta(\varphi, z)$, where $\chi\in C^\infty(\mathbb R)$ be such that
$\supp\chi\subset [0,r_0]$ and $\beta\in C^\infty(\mathbb S^1\times \mathbb R)$.
Then, $B$ is supported in the cylinder $\{r\le r_0\}$ if and only if:
\begin{equation}\label{eq:cylsupp-conditions}
     Z\beta\equiv 0 \quad\text{or}\quad\int_0^{r_0} \chi(t) t^2 dt = 0.
\end{equation}
Indeed, by definition of $\chi$, the primitive $\b$ is supported on $[0,r_0]$ and it holds
\begin{equation*}
    \int_0^{r_0} Z\b(t,\varphi,z) \, tdt = - Z\beta(\varphi,z)\int_0^{r_0}\int_t^{r_0} \chi(s)\,ds\,tdt = Z\beta(\varphi,z)  \int_0^{r_0} \chi(t)\,\frac{t^2}2 dt =0.
\end{equation*}
\end{rmk}

    In the following, we show that far away from the $z$-axis, magnetic fields supported on cylinders are given by Aharonov--Bohm-like vector potentials. This is reminiscent of what happens in the Euclidean case 
    (see, e.g., \cite[Eq.~16]{Krejcirik2013}).	

\begin{proposition}\label{prop:ab-cylinder}
    Let $B = b_1 dr\wedge \omega + b_2 rd\varphi\wedge \omega$ be a horizontal magnetic field whose support is contained in a cylinder $\{r\leq r_0\}$ for some $r_0>0$. 
    Define the map
    \begin{equation}\label{eq:flux}
        \flux{B}(z) := \frac{1}{2\pi} \int_{\R^2} \b(\xi, z)\, d\xi , \qquad z\in \mathbb R.
    \end{equation}
    Then, $\flux{B}(z)$ is independent of $z$. Moreover, letting $\flux{B}=\flux{B}(0)\in \mathbb R$, for an appropriate choice of vector potential $A$, it holds
    \begin{equation}
        A(r,\varphi,z) = \flux{B}\,d\varphi, \qquad \forall r>r_0.
    \end{equation}
\end{proposition}

\begin{proof}
    By the assumption on the support of $B$, it holds $\b_\infty(\varphi,z) = \int_0^{+\infty} b_1(s,\varphi,z)\,ds < +\infty$ for all $(\varphi,z)\in \bS^1\times \bR$. Thus, 
    \begin{equation*}
        \b(t,\varphi,z) = \int_0^t b_1(s,\varphi,z)\,ds - \b_\infty(\varphi,z),\quad t>0.
    \end{equation*}
    Let $f(\varphi,z) = -\int_0^z \b_\infty (\varphi, z')\,dz'$ and
    consider the vector potential $A = A_0+d_Hf$, where $A_0$ is given in the Poincaré gauge as in Proposition~\ref{prop:gauge}. That is, since for $(\varphi,z)\in\bS^1\times\R$ there holds
    \begin{equation*}
    d_Hf(\varphi,z)=-\Phi\left[ \int_0^z \mathfrak b_\infty(\varphi,z')\, dz' \right]rd\varphi=-\left(\int_0^z\partial_\varphi b_\infty(\varphi,z')\;dz'+\frac{r^2}{2}b_\infty(\varphi,z)\right)rd\varphi,
    \end{equation*} 
    we have $A = a(r,\varphi,z)d\varphi$, where
    \begin{equation*}
         a(r,\varphi,z) = \int_0^r \mathfrak b(t,\varphi,z)\,tdt + \frac{r^2}2 \b_\infty(\varphi,z) - \left[\int_0^z\partial_\varphi b_\infty(\varphi,z')\;dz'+\frac{r^2}{2}b_\infty(\varphi,z)\right].
    \end{equation*}
In particular, by construction, $DA = DA_0 = B$, and, using the fact that $\b(t,\varphi,z)=0$ if $t>r_0$, it holds
    \begin{equation}
        \label{eq:alpha-gauge}
        a(r,\varphi,z) = \int_0^{r_0} \b(t,\varphi,z)\,tdt - \int_0^z \partial_\varphi\mathfrak b_\infty(\varphi,z')\, dz', \qquad \forall r>r_0.
    \end{equation}

    Observe that, by \eqref{eq:alpha-gauge}, the function $a(r,\varphi,z)$ is independent of $r>r_0$. Thus, using the fact that $\supp b_1\cup\supp b_2\subset\{r\leq r_0\}$, by Proposition~\ref{prop:DA},  on $\{r>r_0\}$ it holds
    \begin{equation*}
        0 = B = DA = -Za \, d\varphi\wedge\omega 
        \implies
        Za = 0.
    \end{equation*}

    Summing up, for $r>r_0$ we have $A = a(\varphi)\,d\varphi$.
    To prove that $\flux{B}$ is constant, we observe that
    \begin{equation*}
        \begin{split}
            \int_{\bS^1} a(\varphi)\,d\varphi
            & = \int_0^{r_0} \int_{\bS^1} \b(t,\varphi,z)\, d\varphi\,tdt - \int_0^z \int_{\bS^1}\partial_\varphi \b_\infty(\varphi,z)\,d\varphi\,dz\\
            &= \int_0^{+\infty} \int_{\bS^1} \b(t,\varphi,z)\, d\varphi\,tdt \\
            &= 2\pi \flux{B}.
        \end{split}
    \end{equation*}
    Finally, the statement follows  replacing  $A$ with $A + d\psi$, where
    \begin{equation*}
        \psi(r,\varphi,z) = -\int_0^\varphi a(\varphi') \,d\varphi' + \varphi \mathfrak F_B.
        \qedhere
    \end{equation*}
\end{proof}
\begin{remark}\label{rmk:compact-supp}
 If $B$ is compactly supported, say $\supp B \subset \{r\le r_0, |z|\le z_0\}$, we have $\flux{B} =0$. Moreover, using the Poincaré gauge of Proposition~\ref{prop:gauge} it is immediate to observe that in this case also the vector potential $A$ can be chosen to satisfy $\supp A \subset \{|z|\le z_0\}$. 
\end{remark}

\begin{rmk}\label{rem:FB-non-zero}
  Let $B$ be as in Example~\ref{rmk:exa-notbdd}. Then $\flux{B}\neq 0$ if and only if  $\beta$ is independent of $z$.
  In fact, integrating by parts the definition of $\flux{B}$ yields
  \begin{equation*}
  \begin{split}
      \flux{B}
      =-\frac{1}{2\pi}\int_{\bS^1}\int_0^{r_0}b_1(t,\varphi,z)\frac{t^2}{2}\,dtd\varphi
      &=-\frac{1}{2\pi}\int_{\bS^1}\beta(\varphi,z)\;d\varphi\int_0^{r_0}\chi(t)\frac{t^2}{2}\,dtd\varphi.
  \end{split}
  \end{equation*}
  By \eqref{eq:cylsupp-conditions} we then deduce that $B$ is supported in a cylinder if and only if $Z\beta\equiv 0$.
\end{rmk}

\begin{remark}\label{remark:notbdd} The vector potential $A$ of a horizontal magnetic field supported in a cylinder $\{r\leq r_0\}$ is not necessarily bounded on $\bH^1$. In fact, although Proposition~\ref{prop:ab-cylinder} implies that $A$ is bounded outside the cylinder $\{r\leq r_0\}$, by Example~\ref{rmk:exa-notbdd} this does not need to be true inside. 
\end{remark}

We can state now the main result of this section.
This is a more general version of Theorem~\ref{thm:lw1-heis}
presented above.
 
	\begin{theorem}\label{thm:hardy.logarithm}
        Let $B$ be a magnetic field on~$\mathbb{H}^1$
        of the form~\eqref{field.class}  
        and assume that its support is contained
        in the cylinder
        $\{|\xi| \leq r_0\}$ for some $r_0>0$
        and $r_1 > r_0$. 
        If $\flux{B}\notin\mathbb{Z}$,
        there exists a constant~$c(B,r_1)$ 
        dependent on~$B$ and~$r_1$
        such that
        \begin{equation}\label{eq:hardy.logarithm}
          -\Delta_A 
          \geq \frac{c(B,r_1)}{r^2 \left(1+\log_-^2 \frac{r}{r_1}\right)}
          \qquad \mbox{in} \qquad 
          L^2(\mathbb{H}^1)
          \,,
        \end{equation}
	where $\log_- s := \max(0,-\log s)$ for all $s>0$.
	\end{theorem}

	\begin{remark}
	  Theorem~\ref{thm:lw1-heis} follows as a direct 
	  consequence of Theorem~\ref{thm:hardy.logarithm}
	  {and Proposition~\ref{prop:ab-cylinder}}.
	\end{remark}
	\begin{remark}\label{rmk:euclidean.improvement}
	  The strategy of the proof of Theorem \ref{thm:hardy.logarithm} can be easily adapted to the two-dimensional Euclidean setting. This allows to improve \eqref{eq:Hardy-center.mild} in $L^2(\R^2)$ 
	  due to~\cite{LW99} to
	  \begin{equation*}
	      -\Delta_A \geq 
	      \frac{c(B,r_1)}{r^2 \left(1 + \log_-^2 \frac{r}{r_1}\right)}
	      \qquad \mbox{in} \qquad 
              L^2(\R^2)
              \,,
	  \end{equation*}
	  where~$A$ is any vector potential in~$\R^2$ such that $B=dA$
	  with~$B$ is any smooth magnetic field compactly supported in 
	  $\{ x \in \R^2 : \abs{x} \leq r_0\}$ 
	  such that $\Phi_B \not \in \mathbb{Z}$
	   holds true, and $r_1>r_0$.
	\end{remark}
	To prove Theorem \ref{thm:hardy.logarithm} we need the following Lemma.
	    \begin{lem}
      \label{lem:laptev-heis}
        Let $B$ be a magnetic field supported on the cylinder $\{|\xi|\leq r_0\}\subset \mathbb{H}^1$, $r_0>0$, such that $\flux{B}\notin \mathbb{Z}$. Then, there exists a constant $c>0$ such that
        \begin{equation}
            Q_B(u) \ge c \int_{\{|\xi|>r_0\}} \frac{|u|^2}{|\xi|^2}\,dq, \qquad \forall u\in C_c^\infty(\mathbb{H}^1).
        \end{equation}
    \end{lem}

    \begin{proof}
        By Proposition~\ref{prop:ab-cylinder}, the vector potential $A$ can be chosen to be $A=\flux{B}\,d\varphi$ for $r>r_0$, where $\flux{B}\in \mathbb{R}$. Then, using cylindrical coordinates and letting $u(\xi,z) = \int_{\mathbb{R}} e^{i zk} \hat u_k(\xi,k)\,dk$, we have
        \begin{equation}
            \begin{split}
                Q_B(u) 
                &\ge \int_{|\xi|>r_0} |(\dH+i A)u|^2\, d\xi dz\\
                &= \int_\mathbb{R} \int_{\bS^1} \int_{r>r_0} |\partial_r u|^2 + \left| \left( \frac{1}{r}\partial_\varphi + \frac{r}2 \partial_z + i \frac{\flux{B}}{r} \right)u  \right|^2  \, rdr \, d\varphi\,dz \\
                &= \int_{\mathbb{R}}  \int_{\bS^1} \int_{r>r_0} |\partial_r \hat u_k|^2 + \left| \left( \frac{1}{r}\partial_\varphi + i\left(\frac{rk}2 + \frac{\flux{B}}{r}\right) \right)\hat u_k  \right|^2  \, rdr \, d\varphi \, dk\\
                &=: \int_{\mathbb{R}} \hat Q_k(\hat u_k)\,dk.
            \end{split}
        \end{equation}
        Here, $\hat Q_k$ is defined by the last equality, and depends on $r_0$. The main observation is that $\hat Q_k$ is connected to the Euclidean magnetic Dirichlet form in $\mathbb{R}^2$ associated with the vector potential $A_k$, whose expression in polar coordinates is $A_k = \left(\frac{r^2k}2 + {\flux{B}}\chi_{\{r>r_0\}}\right) d\varphi$.
        
        We will now apply a slightly modified version of \cite[Thm.~1]{LW99} to $\hat Q_k$, in order to obtain the existence of $C>0$, independent of $k$, such that
        \begin{equation}\label{eq:laptev-heis}
            \hat Q_k(v) \ge C \int_{|\xi|>r_0}\frac{|v|^2 }{|\xi|^2}\,d\xi, \qquad \forall v\in C^\infty_c(\mathbb{R}^2).
        \end{equation}
        Clearly this will complete the proof by Plancherel identity.

        Write $v(r,\varphi) = \sum_{m\in\bZ} v_m(r) e^{im\varphi}$. Then, letting $\lambda_{m,k}(r) = (m+r^2 k/ 2 +\flux{B})^2$, we have
        \begin{equation}\label{eq:boh1}
            \begin{split}
                \hat Q_k(v) 
                &=\sum_{m\in\bZ} \int_{r>r_0} \left(|\partial_r v_{m}|^2 + \lambda_{m,k}(r) \frac{|v_{m}|^2}{r^2}   \right) \, rdr \\
                &\ge \sum_{m\in\bZ}\int_{r>r_0} \lambda_{m,k}(r) \frac{|v_{m}|^2}{r^2}  \, rdr.
            \end{split}
        \end{equation}
        If $k=0$, then $\lambda_{m,0}(r)\ge (\operatorname{dist}(\flux{B},\mathbb{Z}))^2$, and \eqref{eq:laptev-heis} follows. If $k\neq 0$, observe that $\lambda_{m,k}(r) = \lambda_{m,\operatorname{sgn}(k)}(\sqrt{|k|}r)$. Thus, the change of variables $s = \sqrt{|k|}r$ and the first equality of \eqref{eq:boh1} yield that in order to prove \eqref{eq:laptev-heis} it suffices to obtain existence of $C>0$ such that for all $s_0>0$ and $\varrho\in \{-1,+1\}$ it holds:
        \begin{equation}\label{eq:laptev-heis-2}
            \hat Q^\varrho(w):=\int_{s>s_0} \left(| w'|^2 + \lambda_{m}^\varrho(s) \frac{|w|^2}{s^2}   \right) \, sds \ge C \int_{s>s_0} \frac{|w|^2}{s^2}\,sds,
        \end{equation}
       for all $m\in \mathbb{Z}$, and $w\in C^\infty(\mathbb{R}_+)$ such that $\supp w\subset (0,\bar r]$ for some $\bar r>0$.
    Here, we let $\lambda_{m}^{\pm}(s) = (m \pm {s^2} / {2}+\flux{B})^2$, and omitted the dependence of $\hat Q^\varrho$ on $s_0$, to ease the notation.

    We now present an argument for the case $\varrho = -1$, the other being analogous up to some sign changes.
    Let $\flux{B} = m_0+\gamma$ for  $m_0\in \mathbb{Z}$ and $\gamma\in (-{1} / {2}, {1} /{2}]$, $\gamma\neq 0$, and fix  $\varepsilon \in (0, |\gamma|/ 2)$.  Moreover, for any $\ell\in \mathbb{N}$ we define the following intervals:
        \begin{equation*}
            I_{\ell} = (\alpha_{\ell},\beta_{\ell}),
        \quad\text{where}\quad
        \alpha_{\ell} = \sqrt{{2(\ell+\gamma)-\varepsilon}},\quad \beta_{\ell}=\sqrt{2(\ell+\gamma)+\varepsilon},
        \end{equation*}
        where we let by convention $I_0=\varnothing$ if $\gamma<0$.
        One checks that $\lambda_{m}^-(s)<\varepsilon^2 / 4$ if and only if $m+ m_0\ge 0$ and $s\in I_{m+m_0}$,
%
        and that there exists $\Lambda = \Lambda(\varepsilon,\gamma)$ such that
        \begin{equation}\label{eq:bound-up}
            \beta_{\ell} \le \Lambda \alpha_\ell, 
            \qquad
            |I_{\ell}| \le \Lambda \alpha_{\ell}, \qquad \forall \ell\in\mathbb{N}.
        \end{equation}

        If $m< -m_0$, then $\lambda_m^-(s)\ge \varepsilon^2/4$ for all $s>0$, and we immediately get
        \begin{equation}
            \int_{s>s_0} \frac{|w|^2}{s^2} sds \le \frac{4}{\varepsilon^2} \int_{s>s_0} \lambda_m^-(s) \frac{|w|^2}{s^2}\,sds
            \le \frac 4{\varepsilon^2}\hat Q^-_{s_0}(w).
        \end{equation}
        Let now $\ell:=m + m_0 \ge 0$. If $I_\ell\cap (s_0,+\infty)=\varnothing$, we can proceed as above. If this is not true, we choose a cut-off function $\psi:\mathbb{R}_+\to [0,1]$ such that $\psi^{-1}(1)=\overline{I_{\ell}}$, $\supp\psi \subset (\alpha_\ell-|I_\ell|/2, \beta_\ell+|I_\ell|/2))$, and $|\psi'|\le c|I_\ell|^{-1}$.
        Then, $\lambda_m^-(s)\ge \varepsilon^2 /4$ if $\psi(s)\neq 1$ and we obtain              
        \begin{equation}\label{eq:first-step-lap}
            \begin{split}
              \frac12\int_{s>s_0} \frac{|w|^2}{s^2}\,sds
                &\le  \int_{s>s_0} \frac{|\psi w|^2}{s^2}\,sds + \int_{s>s_0}|1-\psi|^2 \frac{|w|^2}{s^2}\,sds \\
                &\le  \int_{s>s_0} \frac{|\psi w|^2}{s^2}\,sds +  \frac 4{\varepsilon^2}\hat Q^-(w).
            \end{split}
        \end{equation}
        Hence, we are left to bound the first term on 
        the right-hand side of
        the above. In order to do this, we recall the following
        weighted Poincaré inequality, for $0\le \alpha<\beta$ (see \cite{brezis1997blow}):
        \begin{equation*}
            \int_\alpha^\beta |u|^2\,rdr \le \frac{|\beta-\alpha|^2}{2}\int_\alpha^\beta|u'(r)|^2\,rdr, \qquad \forall u\in W^1(\alpha,\beta),\, u(\beta)=0.
        \end{equation*}
        Then, since $|\supp\psi\cap\{s>s_0\}| \le 2|I_\ell|$ and
        $|\psi'|\le c|I_\ell|^{-1}$,  we have,
        \begin{equation*}
            \begin{split}
                \frac 14 \int_{s>s_0} |\psi w|^2\,sds 
                &\le \frac{|I_\ell|^2}2 \int_{\supp\psi\cap \{s>s_0\}}|(\psi w)'|^2\,sds\\
                &\le {|I_\ell|^2} \int_{s>s_0}| w'|^2\,sds + c^2\int_{\supp\psi\setminus I_\ell} |w|^2\,sds\\
                &\le |I_\ell|^2 \hat Q^-(w) + c^2 \left(\beta_\ell+\frac{|I_\ell|}2\right)^2 \int_{\supp\psi\setminus I_\ell} \frac{|w|^2}{s^2} \,sds \\
                &\le \left(|I_\ell|^2 + \frac{4c^2}{\varepsilon^2}\left( \beta_\ell + \frac{|I_\ell|}{2} \right)^2  \right) Q^-(w).
            \end{split}
        \end{equation*}
        Here, we used that $\lambda_m^-(s)\ge \varepsilon^2 /4$ on $\supp\psi\setminus I_\ell$ and that $\sup (\supp\psi\setminus I_\ell)\le \beta_\ell + |I_\ell|/2$.
        Finally, thanks to \eqref{eq:bound-up} we have
        \begin{equation*}
            \int_{s>s_0} \frac{|\psi w|^2}{s^2}\,sds \le \frac{1}{\left( \alpha_\ell - |I_\ell|/2 \right)^2 } \int_{\supp\psi} |\psi w|^2\,sds 
            \le   C Q^-(w).
        \end{equation*}
        Observe that the constant $C$ depends only on $\Lambda=\Lambda(\varepsilon,\gamma)$ and is thus independent of~$s_0$. Together with \eqref{eq:first-step-lap} this proves \eqref{eq:laptev-heis-2}, completing the proof of the statement.
    \end{proof}
	
\begin{proof}[Proof of Theorem \ref{thm:hardy.logarithm}]
Thanks to Proposition~\ref{prop:core}, it is enough to prove the thesis for $u\in C^\infty_c(\mathbb{H}^1\setminus\cZ)$.
    Also, by Lemma~\ref{lem:laptev-heis}, it suffices to show that 
  \begin{equation}\label{eq:unopiu}
            Q_B(u) \ge c\int_{\{|\xi|<r_0\}} \frac{|u|^2}{|\xi|^2 (\log^2
              \frac{|\xi|}{r_1} + 1)}\,dq, \qquad
            \forall u \in C^\infty_c(\mathbb{H}^1\setminus \cZ).
        \end{equation}
  Thanks to 
  \eqref{eq:sublap-cyl} and the Fubini theorem, we write $Q_B(u)$
  in cylindrical coordinates: 
  \begin{equation}\label{eq:quadratic.form.radial}
    Q_B(u) 
    =  
    \int_{0}^{+\infty} \int_{\bS^1 \times \R}
    \left[\abs{\partial_r u}^2 
      + \frac{\abs{( \partial_{\varphi}
          + \frac{r^2}{2}\partial_z + i \alpha(r,\varphi,z))u}^2}{r^2}\right]
    \, r dr d\varphi dz.
  \end{equation}
  Recall the following one-dimensional Hardy-type inequality 
  (see, e.g., \cite[Prop.~2.4]{cassano2018self} for a proof):
        \begin{equation}\label{eq:Hardy.radial}
          \int_{0}^{+\infty} \frac{|f|^2}{r^2\log^2\frac{r}{r_1}}\,rdr \le 4\int_0^{+\infty}|f'|\,rdr, \qquad \forall f\in C^\infty_c(\mathbb{R}_+\setminus\{r_1\}).
        \end{equation}
        Let then $I\subset (r_0,+\infty)$ be a bounded open interval
        such that $r_1\in I$, and consider a cut-off function
        $\psi:\mathbb{R}_+\to [0,1]$ such that $\psi\equiv 0$ on a
        neighborhood of $r_1$ and $\psi^{-1}(1) =
        \mathbb{R}_+\setminus I$.
Writing
$ u = \psi u + (1-\psi) u$ we have that
\begin{equation*}
  \begin{split}
    &\int_{\bS^1 \times \R} \int_0^{+\infty} \frac{|u|^2}%
    {r^2 \big(\log^2 \frac{r}{r_1}+1\big)} r \, dr d\varphi dz
    \\ &\leq
    \int_{\bS^1 \times \R} \left[ 2 \int_0^{+\infty}
      \frac{|\psi u|^2}{r^2 \log^2\frac{r}{r_1}} r \, dr + 2
      \int_0^{+\infty} \frac{|(1- \psi) u|^2}{r^2} r \, dr \right]
    d\varphi dz.
    \\ & \leq
    \int_{\bS^1 \times \R}
    \left[ 8 \int_0^{+\infty} |\partial_r(\psi u)|^2 r \, dr
      + 2 \int_I \frac{|u|^2}{r^2} r \, dr
    \right]
    d\varphi dz,
  \end{split}
\end{equation*}
where in the last inequality we have used \eqref{eq:Hardy.radial},
the fact that $\psi u$ is supported outside the set
$\{r_1\}\times \bS^1 \times \R$, and that $1-\psi$ is supported in $I$.
We immediately have that
\begin{equation}\label{eq:to.be.added}
  \begin{split}
    &\int_{\bS^1 \times \R} \int_0^{+\infty} \frac{|u|^2}%
    {r^2 \big(\log^2 \frac{r}{r_1}+1\big)} r \, dr d\varphi dz
    \\ & \leq
  \int_{\bS^1 \times \R}
  \left[ 16 \int_0^{+\infty}
    |\partial_r u|^2 r \, dr
    +
    \left(16 (\sup I)^2 \norm{\psi'}_{\infty}^2 + 2\right)
    \int_I \frac{|u|^2}{r^2} r \, dr
  \right]
  d\varphi dz.
\end{split}
\end{equation}
Thanks to the explicit expression \eqref{eq:quadratic.form.radial} and
Lemma~\ref{lem:laptev-heis}
we conclude that for some $C>0$
\begin{equation*}
\int_{\bH^1} \frac{|u|^2}%
    {\abs{\xi}^2 \big(\log^2 \frac{\abs{\xi}}{r_1}+1\big)} dq
     \leq
     C \, Q_B(u),
   \end{equation*}
that gives \eqref{eq:unopiu} and concludes the proof.
\end{proof}

\section{Magnetically improved Hardy inequalities from the origin}
%
The main objective of this section is to improve
the Hardy-type inequality~\eqref{eq:GL}
due to Garofalo and Lanconelli \cite{GarofaloLanconelli1990}.
In particular, we establish Theorem~\ref{thm:impro},
showing the subcriticality of the operator
\begin{equation*}
    \co := -\Delta_A  - \frac{r^2}{\rho^4}
\end{equation*}
for Aharonov--Bohm or mild magnetic fields.
 This is the content of Section~\ref{ss:impro},
  where we also comment on the relations with the analogous theory 
  in the Euclidean case.

Before, in Section~\ref{ss:follandstein}, 
we prove Theorem~\ref{thm:GL-FS}
containing a Hardy-type inequality for 
the Folland--Stein operator~\eqref{Folland}.
We use this result to deduce 
the quantitative lower bound~\eqref{eq:impro-rot-intro}
for the operator~$P_A$.

\subsection{Quantitative improvement via Aharonov--Bohm vector potentials}\label{ss:follandstein}

Let $A$ be an Aharonov--Bohm potential on $\mathbb H^1\setminus \cZ$. By Proposition~\ref{prop:AB}, we can assume $A = \alpha \, d\varphi$ for 
$\alpha\in \mathbb R \setminus \bZ$. 
Henceforth, we denote the associated magnetic form by $Q_\alpha$. 

For $\alpha\in \mathbb R$, 
we consider the Folland--Stein operator~\eqref{Folland}
due to~\cite{follandEstimates1974}.
The following proposition highlights the connection between 
the Aharonov--Bohm potentials and~$\cL_\alpha$.

\begin{lem}\label{l:quadform}
For any $\alpha\in\bR$
and \(u\in C^\infty_c(\mathbb H^1\setminus\mathcal Z)\),
we have
\begin{equation}
 Q_\alpha (u) = \langle \cL_\alpha u,u \rangle + \alpha^2 \int_{\mathbb H^1} \frac{|u|^2}{r^2}dp + 2\alpha\langle -i \partial_\varphi u, r^{-2} u\rangle.
\end{equation}
\label{lem:decomp}
\end{lem}

\begin{proof}
Compute
\begin{equation}
\begin{split}
| (d_H+i \alpha\,d\varphi)u | ^2 
&= | R u | ^2 + \left| \Phi u + i\alpha\frac{u}{r} \right|^2 \\
&= | d_H u |^2 + 2\alpha\Re \left( (-i\Phi u) \overline{\frac{u}{r}} \right) + \alpha^2\frac{|u|^2}{r^2}.
\end{split}
\end{equation}
The statement follows by integrating the above and observing that, since \(-i\Phi\) is a symmetric operator, it holds \(\langle-i\Phi u, r^{-1}u\rangle\in \mathbb R\). 
\end{proof}

We now prove Theorem~\ref{thm:GL-FS},
which contains a sharp Hardy inequality for the Folland--Stein operator.
This will immediately imply an improvement of the Garofalo--Lanconelli Hardy inequality for Aharonov--Bohm magnetic potentials, under a symmetry assumption.  
 
\begin{proof}[Proof of Theorem~\ref{thm:GL-FS}]
We define 
\begin{equation*}
    \rho_\alpha := \rho w_\alpha
    \qquad\mbox{and}\qquad 
    w_\alpha(r,\varphi,z) := \exp\left({i \frac{\alpha}{2} \arctan \frac{4z}{r^2}}\right).
\end{equation*}
Observe that $\rho_\alpha^{-2}$
is proportional to the fundamental solution for $\cL_\alpha$, $\alpha\neq 2n+1$, $n\in\mathbb Z$, as shown in \cite{follandEstimates1974}.
Let $u \in C_c^{\infty}(\bH^1 \setminus \cZ)$ and set $v = u \rho_\alpha$. We  compute the integrand of the quadratic form associated with $\cL_\alpha$. This yields
\begin{align*}
    |\nH u|^2 
    &= \frac{|\nH v|^2}{\rho^2} + \frac{|v|^2}{\rho^4}|\nH \rho_\alpha|^2-2\Re\left(  \frac{\bar v\nH v\cdot \nH \overline{\rho_\alpha}}{\rho^2\overline{\rho_\alpha}}\right), \\
    (\partial_z u)\bar u &=
     \frac{\bar v \partial_z v}{\rho^2} - \frac{|v|^2}{\rho^2\rho_\alpha}\partial_z\rho_\alpha.
\end{align*}
We recall that $\nH u =(Xu)X+(Yu)Y$ is the horizontal gradient of $u$.
Direct computations show that
\begin{equation}\label{eq:nabla-rho-alpha}
    \nH \rho_\alpha = w_\alpha\left( \nH\rho + i\alpha \nH^\perp \rho \right), 
    \qquad 
    \partial_z \rho_\alpha = w_\alpha \left( \partial_z \rho + 2i\alpha \frac{|\nH \rho|^2}{\rho} \right).
\end{equation}
Here, we denoted $\nH^\perp \rho = (-Y \rho)X + (X\rho) Y$. In particular, we have that 
\begin{equation*}
    |\nH \rho_\alpha|^2 = (1+\alpha^2)|\nH \rho|^2
    \qquad\text{and}\qquad
    i\alpha\frac{|v|^2}{\rho^2\rho_\alpha}\partial_z\rho_\alpha= i\alpha \frac{|v|^2}{\rho^3}\partial_z\rho-2\alpha^2\frac{|v|^2}{\rho^4}|\nH \rho|^2
    .
\end{equation*}
We thus obtain
\begin{equation}\label{eq:GL-form-decomp}
    \langle \cL_\alpha u,u \rangle
    =\int_{\bH^1}\frac{|\nH v|^2}{\rho^2} + (1-\alpha^2)\frac{|v|^2}{\rho^4}|\nH \rho|^2-2\Re\left(  \frac{\bar v\nH v\cdot \nH \overline{\rho_\alpha}}{\rho^2\overline{\rho_\alpha}}\right)
    +\alpha \Im\int_{\bH^1} \frac{\bar v\partial_z v}{\rho^2}.
\end{equation}
Here, we used the fact that
\begin{equation*}
    \Re\int_{\bH^1}\left[\frac{\bar v \partial_z v}{\rho^2} - \frac{|v|^2}{\rho^3}\partial_z\rho\right] = \Im\int_{\bH^1} \frac{|v|^2}{\rho^3}\partial_z\rho = 0.
\end{equation*}

By \eqref{eq:nabla-rho-alpha}, using the fact that $\rho^{-2}$ is proportional to the fundamental solution of the sub-Laplacian $-\sublap$ and that $\abs{v(0)}=0$, we obtain
\begin{equation*}
    -2\int_{\bH^1}\Re\left(  \frac{\bar v\nH v\cdot \nH \overline{\rho_\alpha}}{\rho^2\overline{\rho_\alpha}}\right)\,dp = \alpha \Im \left( \int_{\bH^1} \nH^\perp (\rho^{-2})\cdot (\bar v\nH v)\,dp\right).
\end{equation*}
To conclude the proof we are thus left to show that
\begin{equation*}
    \int_{\bH^1}\frac{|\nH v|^2}{\rho^2} +\alpha \Im \left( \int_{\bH^1} \left(\nH^\perp (\rho^{-2})\cdot (\bar v\nH v)+ \frac{\bar v\partial_z v}{\rho^2}\right)\,dp\right) \ge 0.
\end{equation*}
Integrating by parts the second term on 
the left-hand side and recalling that $[X,Y] = Z$, we have
\begin{equation*}
    \int_{\bH^1} \left( \nH^\perp (\rho^{-2})\cdot (\bar v\nH v)+ \frac{\bar v\partial_z v}{\rho^2}\right) \,dp= \int_{\bH^1} \frac{\nH^\perp \bar v \cdot \nH v}{\rho^2}\,dp.
\end{equation*}
Since $|\alpha|< 1$ and $|\nH^\perp \bar v\cdot\nH v|\le |\nH v|^2$, this completes the proof, except for the optimality.

The sharpness follows as in the Garofalo--Lanconelli case, by considering a smoothing of the functions $\chi_{\{1/k\le \rho\le k\}}\, \rho_\alpha^{-1}$, $k\in\mathbb N$. In particular, we let $u_k:=v_k e^{if_\alpha}$, where $v_k$ is a compactly supported smooth approximation of $\chi_{\{1/k\le \rho\le k\}}\, \rho^{-1}$ and define
\(f_\alpha:=-(\alpha/2) \arctan (4z/r^2)\). 
Since both $v_k,f_\alpha$ are real-valued, we have
\begin{equation*}
\begin{split}
\int_{\mathbb{H}^1}|\nH u_k|^2\;dp&=\int_{\mathbb{H}^1}|\nH v_k+iv_k\nH f_\alpha|^2\;dp
=\int_{\mathbb{H}^1}|\nH v_k|^2+v_k^2|\nH f_\alpha|^2\,dp,\\
-i\int_{\mathbb{H}^1}\bar u_k Z u_k \;dp&=-i\int_{\mathbb{H}^1}(Z v_k+iv_kZ f)v_k\,dp
=\int_{\mathbb{H}^1}v_k^2Z f_\alpha(r,z)\;dp.
\end{split}
\end{equation*}
Here, in the last equation, we have integrated by parts.
We then obtain
\begin{equation*}
    \int_{\bH^1}\bar u_k\cL_\alpha u_k\,dp=
    \int_{\mathbb{H}^1}|\nH v_k|^2+v_k^2(|\nH f_\alpha|^2+\alpha Zf_\alpha)\,dp.
\end{equation*}
Direct computations show that $|\nH f_\alpha|^2+\alpha Zf_\alpha=-\alpha^2 r^2/\rho^4$. Recalling that $v_k$ is a minimizing sequence for 
the Garofalo--Lanconelli Hardy inequality, 
the proof is then concluded by observing the following:
\begin{equation*}
    \lim_{k\to\infty}\frac{\int_{\bH^1}\bar u_k \cL_\alpha u_k\,dp}{\int_{\bH^1}|u_k|^2\frac{r^2}{\rho^4}\,dp}=
    \lim_{k\to\infty}\left(\frac{\int_{\mathbb{H}^1}|\nH v_k|^2-\alpha^2v_k^2\frac{r^2}{\rho^4}\,dp}{\int_{\bH^1}v_k^2\frac{r^2}{\rho^4}\,dp}\right)=1-\alpha^2. \qedhere
\end{equation*}
\end{proof}

\begin{proposition}\label{prop:improvement-GL}
    Let  $A = \alpha d\varphi$ for $\alpha\in\bR$. 
    Write $\alpha = n+\gamma$, where $n\in \mathbb Z$ and $\gamma = \pm d(\alpha,\mathbb Z)$. Then, for any $u\in C^\infty_c(\bH^1\setminus\cZ)$, 
    it holds:
\begin{enumerate}
    \item [i.] If  $\gamma\langle-i\Phi u,u r^{-1} \rangle\ge -\gamma n \int_{\bH^1}\frac{|u|^2}{r^2}\,dp$, then
    \begin{equation*}
        Q_\alpha(u) -\int_{\bH^1}\frac{r^2}{\rho^4}|u|^2\,dp\ge  d(\alpha,\bZ)^2\int_{\bH^1}\frac{|u|^2}{r^2}\,dp.
    \end{equation*}
    In particular, for real-valued $u$ this always holds if $\gamma n \ge 0$ (e.g., if $|\alpha|\le 1/2$).
    \item [ii.] If  $-i\partial_\varphi u = nu$ or if $u(x,y,z) = u(x,y,-z)$, then
    \begin{equation}\label{eq:improved-GL-boh}
        Q_\alpha(u) - \int_{\bH^1}\frac{r^2}{\rho^4}|u|^2\,dp \ge d(\alpha,\bZ)^2\int_{\bH^1}\frac{|u|^2}{r^2}\left( 1-|\nH \rho|^4\right)\,dp.
    \end{equation}
\end{enumerate}
\end{proposition}

\begin{proof}
We start by observing that, performing the gauge transformation $v(r,\varphi,z) = u(r,\varphi,z)e^{i n\varphi}$, we can assume $n=0$, i.e., $|\alpha|\le 1/2$.
    
    The first part of the statement follows immediately from Lemma~\ref{l:quadform} and the Garofalo--Lanconelli Hardy inequality \eqref{eq:GL}. Indeed, since $n=0$, we have that $\alpha\langle-i\Phi u,u r^{-1} \rangle\ge 0$, which implies that
    \begin{equation*}
        \langle \cL_\alpha u, u\rangle + 2\alpha\langle -i\partial_\varphi u, r^{-2}\varphi\rangle \ge \langle -\sublap u, u \rangle \ge \int_{\bH^1} \frac{r^2}{\rho^4}|u|^2\,dp.
    \end{equation*}
    
    The second part of the statement, in the case $\partial_\varphi u= nu =0$, 
    follows from Lemma~\ref{lem:decomp}, Theorem~\ref{thm:GL-FS}, and the fact that
    \begin{equation*}
        \frac{1}{r^2}-\frac{r^2}{\rho^4} = \frac{1}{r^2}\left(1-|\nH\rho|^4\right).
    \end{equation*}
    In the case where $u(x,y,z)=u(x,y,-z)$, the statement follows from Lemma~\ref{lem:impro-phi-v}. Indeed, considering cylindrical coordinates and writing $u(r,\varphi,z)=\sum_k u_k(r,z)e^{ik\varphi}$, the change of variables $z\mapsto -z$ allows to show that
    \begin{equation*}
        \int_{\bR_+\times \bR} \frac{z}{\rho^4}|u_k|^2\eul(\arg(u_k)) \,rdrdz =0, \qquad \forall k\in \bZ.
        \qedhere
    \end{equation*}
\end{proof}

Recall that the Euler vector field $\mathcal E$ is the only vector field such that $\mathcal E(u)=\nu u$ if $u$ is $\nu$-homogeneous with respect to the dilations $(r,\varphi, z)\mapsto (\lambda r,\varphi, \lambda^2 z)$. In particular, $\mathcal E(u) = r\partial_r u + 2z\partial_z u$.

\begin{lem}\label{lem:impro-phi-v}
    Let $|\alpha|\le 1/2$. Then, \eqref{eq:improved-GL-boh} holds
    for all $u\in C^\infty_c(\bH^1\setminus\cZ)$ such that, writing $u(r,\varphi,z)=\sum_{k\in\bZ} u_k(r,z) e^{ik\varphi}$,  it holds that  
    \begin{equation}\label{eq:ass-eul}
        \int_{\bR_+\times \bR} \frac{z}{\rho^4}|u_k|^2\left( \frac{z}{r^2} - \operatorname{sgn}(k) \eul(\arg(u_k))\right) \,rdrdz \ge 0,
    \end{equation}
    for any $k\in\bZ$, such that $|k|\ge 2$ or $|k|=1$ and $k\alpha\ge 0$.
\end{lem}

\begin{proof}
Let us pass in Fourier 
with respect to the $\varphi$ variable in Lemma~\ref{l:quadform}. Since
\begin{equation}
    \Phi = \bigoplus_{k\in \mathbb Z} \Phi_k, \qquad \Phi_k = \frac r2\partial_z +\frac{ik}r,
\end{equation}
we have that 
\begin{equation}\label{eq:Q-ak}
    Q_\alpha(u) = \sum_{k\in \bZ} \hat Q_{\alpha,k} (u_k), \qquad \hat Q_{\alpha,k} (v) = \langle \tilde\cL_{\alpha+k} v , v \rangle+(k+\alpha)^2\int_{\bR_+\times \bR}\frac{|v|^2}{r^2}\,d\mu.
\end{equation}
Here, we let $d\mu := rdrdz$, and denoted by $\tilde \cL_\gamma$ the $0$-th Fourier component of the operator~$\cL_\gamma$. That is,
\begin{equation}
    \tilde \cL_\gamma = R^*R + \Phi_0^*\Phi_0 -i \gamma\partial_z, 
    \qquad \text{in} \qquad 
    L^2(\bR_+\times \bR, d\mu).
\end{equation}
Observe that Theorem~\ref{thm:GL-FS} implies that $\tilde \cL_\gamma \ge (1-\gamma^2)r^2/\rho^4$ for any $\gamma\in [-1,1]$.

The case $k=0$ is easily treated. Indeed, applying Theorem~\ref{thm:GL-FS} yields
\begin{equation}
    \hat Q_{\alpha,0}(v) \ge (1-\alpha^2)\int_{\bR_+\times \bR} \frac{r^2}{\rho^4}|v|^2 + \alpha^2\int_{\bR_+\times \bR}\frac{|v|^2}{r^2}.
\end{equation}
Let us consider now the case $k=1$ and $\alpha\le 0$ (the case $k=-1$ and $\alpha\ge 0$ can be treated analogously). In this case, $|\alpha+k| = 1-|\alpha|\in [-1,1]$. Hence, Theorem~\ref{thm:GL-FS} yields
\begin{equation*}
    \begin{split}
    \hat Q_{\alpha,1}(v) 
    &\ge (1-(1-|\alpha|)^2)\int_{\bR_+\times \bR} \frac{r^2}{\rho^4}|v|^2 + (1-|\alpha|)^2\int_{\bR_+\times \bR}\frac{|v|^2}{r^2} \\
    &\ge (1-\alpha^2)\int_{\bR_+\times \bR} \frac{r^2}{\rho^4}|v|^2 + \alpha^2\int_{\bR_+\times \bR}\frac{|v|^2}{r^2}.
    \end{split}
\end{equation*}
Here, in the second inequality we have used the fact that $1/r^2\ge r^2/\rho^4$.

We are left with the case $|k|\ge 2$ or $|k|=1$ and $k\alpha>0$. Observe that, proceeding as in the proof of Theorem~\ref{thm:GL-FS}, and observing that $\nabla^\perp \bar v\cdot \nabla v  = 2\Im(X\bar v \,Yv) = 2\Im(R\bar v\,\Phi v)$, one obtains
\begin{equation}\label{eq:L-ak}
    \langle \tilde\cL_{\alpha+k} v , v \rangle = (1-(\alpha+k)^2)\int_{\bR_+\times \bR} \frac{r^2}{\rho^4}|v|^2\dmu + I(v),
\end{equation}
where 
\begin{equation*}
    I(v) := \int_{\bR_+\times \bR} \frac{|R w|^2 + |\Phi_0 w|^2}{\rho^2}\,\dmu + 2(\alpha+k) \int_{\bR_+\times\bR} \frac{\Im(R\bar w\,\Phi_0 w)}{\rho^2}\,\dmu, \qquad w:=v\rho_{\alpha+k}.
\end{equation*}
Since $|\alpha|\le 1$ and it holds $2\left|\Im(R\bar w \Phi_0 w)\right|\leq |Rw|^2+|\Phi_0 w|^2$, we obtain
\begin{equation*}
    \begin{split}
        I(v) 
        &\geq (1-|\alpha|) \int_{\bR_+\times \bR} \frac{|R w|^2 + |\Phi_0 w|^2}{\rho^2} \,\dmu + 2 k \int_{\bR_+\times\bR} \frac{\Im(R\bar w\,\Phi_0 w)}{\rho^2}\,\dmu\\
        &\geq 2 k \int_{\bR_+\times\bR} \frac{\Im(R\bar w\,\Phi_0 w)}{\rho^2}\,\dmu.
    \end{split}
\end{equation*}

A double integration by parts yields
\begin{equation}\label{eq:lontano}
    \int_{\bR_+\times \bR}\frac{\Im(R\bar w \Phi_0 w)}{\rho^2}\,\dmu
    =\frac{i}{2}\int_{\bR_+\times \bR}\bar w\frac{8z}{\rho^6}(2z\partial_zw+r\partial_rw)\,\dmu
    =4{i}\int_{\bR_+\times \bR}\bar w \eul(w) \frac{z}{\rho^6}\,\dmu.
\end{equation}
Observe that $\rho_{\alpha+k}$ is $1$-homogeneous, and as such $\eul(\rho_\alpha) = \rho_\alpha$.
Hence, by definition of~$w$ and the fact that the above expression is purely real, we obtain
\begin{equation*}
    \int_{\bR_+\times \bR}\frac{\Im(R\bar w \Phi_0 w)}{\rho^2}\,\dmu = -4\int_{\bR_+\times \bR}\Im(\bar v\eul(v))\frac{z}{\rho^4}\,\dmu.
\end{equation*}
Summing up, from \eqref{eq:L-ak} we have that
\begin{equation*}
    \langle \tilde\cL_{\alpha+k} v , v \rangle \ge (1-(\alpha+k)^2)\int_{\bR_+\times \bR} \frac{r^2}{\rho^4}|v|^2\dmu -8k\int_{\bR_+\times \bR}\Im(\bar v\eul(v))\frac{z}{\rho^4}\,\dmu.
\end{equation*}
Plugging the above in \eqref{eq:Q-ak}, and using that $\frac1{r^2} = \frac{r^2}{\rho^4}+\frac{16 z^2}{r^2\rho^4}$, yields
\begin{multline*}
    \hat Q_{\alpha,k}(v) \ge (1-\alpha^2)\int_{\bR_+\times \bR} \frac{r^2}{\rho^4}|v|^2\dmu +\alpha^2\int_{\bR_+\times \bR} \frac{|v|^2}{r^2}\dmu\\ +8\int_{\bR_+\times \bR} \frac{z}{\rho^4}\left( 2k(2\alpha+k) \frac{z}{r^2}|v|^2 - k \Im(\bar v\eul(v))\right) \,\dmu .
\end{multline*}
Due to the range of $k$ and $\alpha$ under consideration, we have that $2k(\alpha+k)\ge |k|$. The statement follows by assumption \eqref{eq:ass-eul} observing that $\Im(v\eul(v))=|v|^2\eul(\arg v)$.
\end{proof}


\begin{remark}\label{remark:xiao}
  In \cite{Xiao2015}, the author claims to prove an improvement of a weighted Hardy inequality in the spirit of Garofalo--Lanconelli inequality \eqref{eq:GL} via the magnetic potential  $A=-(Y\rho/\rho)\,dx+(X\rho/\rho)\,dy$ on $\bH^1\setminus\{0\}$. However, as we show in Appendix~\ref{app:xiao}, the stated result lacks a crucial symmetry assumption on the class of functions under consideration. This agrees with the fact that, due to the trivial cohomology of $\mathbb H^1\setminus\{0\}$, the vector potential $A$ is exact and thus it cannot improve the Hardy inequality on arbitrary smooth functions.
  \end{remark}

\subsection{Localized improvement 
of the Hardy--Garofalo--Lanconelli inequality}\label{ss:impro}
Let $A$ be a magnetic vector potential that is either 
the Aharonov-Bohm potential on $\mathbb H^1\setminus \cZ$ 
(i.e., $A = \alpha d\varphi\mod\omega$, $\alpha\in \mathbb R\setminus\bZ$) 
or smooth. 
We consider the operator $\co$ obtained as the Friedrichs extension of \( -\Delta_A - \frac{r^2}{\rho^4}\) with initial domain 
\(C^\infty_c(\mathbb H^1)\). 
(Observe that $r^2/\rho^4\le 1/\rho^2$ is integrable near the origin.)
The associated form is
\begin{equation*}
\cf(u) := \int_{\bH^1}  \left[|(\dH+iA)u|^2 - \frac{r^2}{\rho^4}|u|^2 \right]\,dq
\,, \qquad
\forall u \in C^\infty_c(\mathbb H^1)
\,.
\end{equation*}
When \(A=0\), the operator \(\co\) is critical, 
due to the sharpness of the Hardy-type inequality \eqref{eq:GL}. Our aim in this section is to show that this is never the case if \(A\neq 0\). Namely, we prove Theorem~\ref{thm:impro}.

We start with a straightforward result. 
\begin{lem}
Let \(\Omega\subset \mathbb H^1\) be an open set
and let \(\cf^\Omega(u)\) be the restriction of \(\cf\) to \(\Omega\), 
that is,
\begin{equation}
\cf^\Omega(u) := \int_{\Omega}  \left[|(\dH+iA)u|^2 - \frac{r^2}{\rho^4}|u|^2 \right]\,dq.
\end{equation}
We let \(\co^\Omega\) be the self-adjoint operator on 
\(L^2(\Omega)\) associated with the closure of \(\cf^\Omega\) 
with domain \(C^\infty_c(\bH^1)\). 
Then, \(\co \ge \co^\Omega\) in the sense of quadratic forms.
\label{lem:lower-bound}
\end{lem}

Recall that \(L^2(\mathbb H^1) = L^2(\Omega)\oplus L^2(\mathbb H^1\setminus \Omega)\), and let \(\Theta\) be the null operator on \(L^2(\mathbb H^1\setminus \Omega)\), i.e., \(\Theta u = 0\) for all \(u\in L^2(\mathbb H^1\setminus\Omega)\).
Then, since \(\co^\Omega\) corresponds to the restriction of \(\co\) to \(\Omega\), with Neumann boundary conditions, 
letting $\tilde{P}_A^\Omega$
be the self-adjoint operator on \(L^2(\Omega)\) associated with the closure of \(\cf^\Omega\) with domain \(C^\infty(\Omega)\), we have
\begin{equation}
{\co^\Omega} = \tilde{P}_A^\Omega \oplus \Theta.
\end{equation}
Thus, henceforth, we will identify \({\co^\Omega}\) 
with $\tilde{P}_A^\Omega$, with abuse of notation.
The main observation is then the following. 
\begin{proposition}
\label{prop:discrete}
Let \(\Omega\subset\bH^1\) be a bounded open set with Lipschitz boundary. 
Let $A$ be a magnetic vector potential that is either Aharonov--Bohm on $\Omega\setminus \cZ$ or smooth. Then, \({\co^\Omega}\) has discrete spectrum.
\end{proposition}

Let $B_\delta:=\{q\in \bH^1 : \rho(q)<\delta\}$, for $\delta>0$.
We will need the following ``unweighted'' Hardy inequality 
(see, e.g., \cite[Lem.~2.1]{BCX05}):
\begin{equation}\label{eq:hardy-unw}
\int_{B_\delta} \frac{|v|^2}{\rho^2}\,dq \le C \int_{B_\delta}|\dH v|^2\,dq, \qquad v\in W^1(B_\delta).
\end{equation}
The interest here with respect to
 Hardy inequality~\eqref{eq:GL} is the absence of weight 
on the left-hand side. 
However, this comes at the cost of not having an explicit constant \(C>0\).

In the following, we denote by $W^1_A(\Omega)$ 
the form domain of $-\Delta_A$ in $L^2(\Omega)$
with Neumann boundary conditions. That is,
\begin{equation*}
    W^1_A(\Omega) := \left\{ u\in L^2(\Omega) :\: Q_A(u)<+\infty \right\}.
\end{equation*}
The following embedding result for $W^1_A(\Omega)$ will be crucial. Let us observe that the Aharonov--Bohm part of Lemma~\ref{l:RK}  relies on the validity of the improved Hardy inequality from the center presented in  Theorem~\ref{thm:lw2-heis}.
\begin{lem}
\label{l:RK}
Let $\Omega \subset \mathbb H^1$ be a bounded domain 
with Lipschitz boundary.
Let $A$ be a magnetic vector potential on $\Omega$ 
that is either of Aharonov--Bohm type on $\Omega\setminus \cZ$ or smooth. Then, $W^1_A(\Omega)$ is compactly embedded into $L^2(\Omega)$. 
\end{lem}

\begin{proof} We show that $W^1_A(\Omega)\subset W^1(\Omega)$. The statement then follows by the Rellich--Kondrachov theorem for sub-Laplacians \cite[Lem.~4.3]{FPR20}. 

Considering cylindrical coordinates, we have that $v\in C^\infty(\Omega)\cap L^2(\Omega)$ belongs to $W^1_A(\Omega)$ if
\begin{equation}
Q_A(v)=\int_\Omega |\partial_rv|^2r\,dr d\varphi dz+\int_\Omega\left|\frac{1}r\partial_\varphi v+\frac{r}{2}\partial_zv-iA v\right|^2r\,dr d\varphi dz<+\infty.
\end{equation}
In particular, if $v\in L^2(\Omega)$ is such that $Q_A(v)<+\infty$, we immediately have that $\|\partial_rv\|_{L^2(\Omega)}<\infty$. 
To conclude the proof we are left to show that $\|\Phi v\|_{L^2(\Omega)} <\infty$, where $\Phi$ is defined in \eqref{eq:RPhi}. 

In the case where $A$ is smooth on $\Omega$, the statement follows by observing that $\|A\|_{L^\infty(\Omega)}<\infty$, thus implying $\|\Phi v\|_{L^2(\Omega)} \le \sqrt{Q_A(v)} +\|Av\|_{L^2(\Omega)}<\infty$.

Otherwise, if $A$ is an Aharonov--Bohm potential, by Proposition~\ref{prop:AB} we assume without loss of generality that 
$A= \alpha \, d\varphi$, 
with $\alpha\in \mathbb R\setminus\bZ$, so that
\begin{equation}
Q_A(v)=\int_\Omega |\partial_rv|^2r\,dr d\varphi dz+\int_\Omega \frac{|\partial_\varphi v+\frac{r^2}{2}\partial_zv-i\alpha v|^2}{r^2}r\,dr d\varphi dz.
\end{equation}
Then, by Theorem~\ref{thm:lw2-heis}, we have that
\begin{equation}
\left\|\frac{v}{r}\right\|_{L^2(\Omega)}^2\leq\frac{Q_A(v)}{\operatorname{dist}(\alpha,\mathbb Z)^2}.
\end{equation}
In particular, it holds that $\|\Phi v\|\leq \sqrt{ Q_A(v)}+ \alpha\|v/r\|_{L^2(\Omega)}<+\infty$, thus concluding the proof.\end{proof}

\begin{proof}[Proof of Proposition~\ref{prop:discrete}]
The proof is an adaptation of the Euclidean one, presented e.g.\ in \cite{Escobedo1987}. 
Let \(L^2:=L^2(\Omega)\) and \(W_A^1:=W^1_A(\Omega)\). 
We start by recasting the problem in a weighted space. Let \(\phi = \rho^{-2}\) and consider the unitary transformation \(T: L^2\to L^2(\phi)\) defined by \(Tv = \phi^{-1/2}v\). Here, we denoted by \(L^2(\phi)\) the space of functions on \(\Omega\) that are square integrable with respect to the measure \(\phi(q)\,dq\).

Observe that by~\eqref{eq:m-subLapl}, for any $u,v\in C^\infty(\Omega)$ we have that
\begin{equation*}
    -\Delta_A(uv)=-u\Delta_Av-2\nabla_{\!A} v\cdot\nH u-v\sublap u,
\end{equation*}
where $\nabla_{\!A}:=\nH+iA$ with $\nH$ being the horizontal gradient and $A$ being identified with the horizontal vector field $A=A_xX+A_yY$.
Letting $L_\phi = T\circ \co^\Omega \circ T^{-1}$, we deduce that for any $u\in C^\infty(\Omega)$ one has 
\begin{equation*}
    L_\phi u 
    =-\Delta_Au-\nabla_{\!A}u\cdot(\phi^{-1}\nH\phi)\\
    =-\frac{1}{\phi}\left[X_A(\phi X_Au)+Y_A(\phi Y_Au)\right].
\end{equation*}
Here, we used the fact that $\Delta(\phi^{1/2})=-r^2/\rho^5$, which follows from direct computations.
Thus, \(\co^\Omega\) is unitarily equivalent to the operator \(L_\phi\) on \(L^2(\phi)\), and hence it will suffice to show the discreteness of the spectrum for the latter. 
In order to do so, we set to show that the form domain of \(L_\phi\) compactly embeds in \(L^2(\phi)\). 

The form associated with \(L_\phi\) on \(L^2(\phi)\) is
\begin{equation*}
\ell_\phi(u) := \int_\Omega |(\dH+iA) u|^2\phi\,dq.
\end{equation*}
We let \(W^1_A(\phi)\) be the form domain of \(\ell_\phi\) endowed with the norm 
$$
  \|u\|_{W^1(\phi)} := \sqrt{\|u\|_{L^2(\phi)}^2 + \ell_\phi(u)}
  \,.
$$
Observe that, due to the boundedness of \(\Omega\), there exists \(a>0\) such that \(\phi\ge a\) on \(\Omega\). Thus, \(W^1_A(\phi)\subset W^1_A\). Recall that, by Lemma~\ref{l:RK}, \(W_A^1\) compactly embeds in \(L^2\).
If \(0\notin \Omega\), then \(\phi\) is also bounded from above, and thus \(W^1_A(\phi)=W^1_A\) and \(L^2(\phi)=L^2\), which easily yields the statement. Henceforth we will thus assume \(0\in \Omega\).

Let \((u_k)_k\subset W^1_{A}(\phi)\) be weakly convergent in \(W^1_{A}(\phi)\) to \(u\). We can assume \(\|u_k\|_{W^1_A(\phi)}\le 1\). Since \(W_{A}^1(\phi)\subset W_{A}^1\), and the latter is compactly embedded in \(L^2\) by assumption, we have that \(u_k\rightarrow u\) in \(L^2_{\text{loc}}\).

Fix $\varepsilon>0$. Let \(\delta\in(0,1)\) be a constant to be fixed later, and consider a cut-off function \(\chi:\Omega\to [0,1]\) such that \(\supp\chi\subset B_\delta\) and \(\chi \equiv 1\) on \(B_{\delta/2}\). Let us define \(u^1_k := \chi u_k\), \(u_k^2 := (1-\chi)u_k\), \(u^1:=\chi u\), and \(u^2:=(1-\chi)u\). Observe that \(\|u^i_k\|_{W_A^1(\phi)}\le 2\) for \(i=1,2\) and \(k\in\mathbb N\).

Since the \(u^2_k\)'s are supported outside \(B_{\delta/2}\), we have that \(\phi\le 4/\delta^2\) on their support. Thus, we readily obtain
\begin{equation*}
\int_{\Omega} |u_k^2-u^2|^2 \phi\,dq \le \frac{4}{\delta^2}\int_{\Omega\setminus B_{\delta/2}} |u_k-u|^2.
\end{equation*}
In particular, since \(u_k\rightarrow u\) in \(L^2_{\text{loc}}\),  there exists $\bar k \in \mathbb N$ also depending on $\delta$ such that for $k\geq \bar k$ we have $\|u_k^2-u^2\|_{L^2(\phi)}^2<\varepsilon/2$. 

To treat the \(u^1_k\)'s, we apply the  ``unweighted'' Hardy inequality~\eqref{eq:hardy-unw}. 
Indeed, recalling that \(\phi = \rho^{-2}\), this allows to estimate, independently on \(k\in\mathbb N\), 
\begin{equation*}
\int_{\Omega} |u_k^1-u^1|^2 \phi\,dq \le C \int_{B_\delta}|\dH (u_k^1-u^1)|^2\,dq \le \delta^2C \int_{B_\delta}|\dH (u_k^1-u^1)|^2\phi\,dq \le 2\delta^2C.
\end{equation*}
Therefore, choosing $\delta<\sqrt{\varepsilon/(4C)}$ the last inequality implies $\|u_k^1-u^1\|_{L^2(\phi)}^2<\varepsilon/2$.
Finally, we have proved that with this choice of $\delta$, for $k\geq\bar k$ we have that 
\begin{equation*}
\begin{split}
\int_{\Omega} |u_k-u|^2 \phi\,dq
&\le \int_{\Omega} |u_k^1-u^1|^2 \phi\,dq  + \int_{\Omega} |u_k^2-u^2|^2 \phi\,dq
\le \varepsilon
\end{split}
\end{equation*}
This completes the proof.
\end{proof}

We will also need the following result, about the ground state of the unperturbed operator $P_0^\Omega$.

\begin{proposition}\label{prop:simple}
    Let $\Omega\subset \bH^1$ be a bounded and connected open set with Lipschitz boundary. Then, $0$ is the smallest eigenvalue of $P_0^\Omega$, it is simple, and the corresponding  ground state is proportional to $\rho^{-1}$.
\end{proposition}

\begin{proof}
{We start by claiming that $\rho^{-1}\in {\dom(\cfo^\Omega)}$. To prove the claim, we approximate $\rho^{-1}$ by a sequence of smooth functions compactly supported in $\bH^1$ which is Cauchy in the topology of the form $\cfo^\Omega$.

If $0\notin \Omega$, such a sequence can be chosen to be equal to $\rho^{-1}$ in $\Omega$. If $0\in \Omega$, we let $N,M\in \mathbb N$ be so that $\{\rho<1/N\}\subset\Omega\subset \{\rho<M\}$. For any $n\in\mathbb N$, $n\geq N$  and $\varepsilon>0$, we let $u_n^\varepsilon\in C_c(\bH^1)$ be defined as 
\[
u_n^{\varepsilon}(q)=\begin{cases}
0,&\text{if }\rho(q)<1/n^2\text{ or } \rho(q)>M^2,\\
\psi(\log_n(n^2\rho(q)))n^{2+\varepsilon}\rho(q)^{1+\varepsilon},&\text{if }\rho(q)\in[1/n^2,1/n]\\
\rho^{-1}(q),&\text{if }\rho(q)\in(1/n,M],\\
\psi(\log_M(M^2\rho^{-1}))/M,&\text{if }\rho(q)\in (M,M^2].
\end{cases}
\]
where $\psi:[0,1]\to[0,1]$ is a smooth function such that $\psi=0$ in a right neighborhood of~$0$ and $\psi=1$ in a left neighborhood
of~$1$.

Recalling that $|\nabla\rho| = r/\rho$, it is immediate to check that
\[
\int_{\Omega}  \left[\left|\nH(\rho^{-1})\right|^2 - \frac{r^2}{\rho^4}\left|\rho^{-1}\right|^2 \right]\,dq=0.
\]
Observing also that  $\rho^{-1}\in L^2(\Omega)$, straightforward computations then yield the claim since $u_n^\varepsilon$ is a Cauchy sequence in the topology of $\cfo^\Omega$ so that, as $n\to +\infty$, and $\varepsilon\to 0$ we obtain
\[
\begin{split}
\|u_n^\varepsilon-\rho^{-1}\|_{L^2(\Omega)}\to 0,\quad \text{and} \quad \cfo^\Omega(u_n^\varepsilon)\to 0.
\end{split}
\]}

By Proposition~\ref{prop:discrete}, the min-max principle, and the non-negativity of the operator, this proves that the smallest eigenvalue of $P_0^\Omega$ is $0$ and that $\rho^{-1}$ is an associated eigenfunction.

We now provide an argument for the simplicity. 
This is based on the following identity, 
holding for any $u\in {\dom(\cfo^\Omega)}$,
\begin{equation}\label{eq:cfo}
    \cfo(u) = \int_{\bH^1}\left( \left| R u + \frac{r^3}{\rho^4} u \right|^2 + \left| \Phi u+\frac{4rz}{\rho^4} u\right|^2 \right)\,dq.
\end{equation}
Indeed, by developing the squares on the right-hand side
above, 
we obtain
\begin{equation}
    \begin{split}
    \int_{\bH^1}\bigg(\left| R u + \frac{r^3}{\rho^4} u \right|^2 &+ \left| \Phi u+\frac{4rz}{\rho^4} u\right|^2\bigg)\,dq\\
    &= \cfo(u) + 2\left(\int_{\bH^1} \frac{r^2}{\rho^4}|u|^2\, dq+ \Re\int_{\bH^1}\left[ \frac{r^3}{\rho^4} Ru + \frac{4rz}{\rho^4}\Phi u\right] \bar u\, dq \right).
    \end{split}
\end{equation}
Equation \eqref{eq:cfo} follows by observing that the last term in the above vanishes, as can be checked by integration by parts.

Finally, to complete the proof it suffices to observe that $\rho^{-1}$ is the only solution (up to multiplicative constants) of
\begin{equation*}
    Ru + \frac{r^3}{\rho^4} u = 0 
    \quad\text{and}\quad
    \Phi u+\frac{4rz}{\rho^4} u =0.
    \qedhere
\end{equation*}
\end{proof}

We are now ready to prove Theorem~\ref{thm:impro}.
\begin{proof}[Proof of Theorem~\ref{thm:impro}]
By Lemma \ref{lem:lower-bound} it holds \(\co\ge \co^\Omega\), and by Proposition~\ref{prop:discrete} we have that \(\co^\Omega\) has discrete spectrum. Let \(\lambda\) be the smallest eigenvalue of \({\co^\Omega}\), then \(\co \ge \lambda \chi_\Omega\).
Let us shows that \(\lambda>0\) if $A$ is of either type (i) or (ii).
Indeed, this yields the statement with $c(A,\Omega)=\lambda$. Observe that, due to gauge invariance, $\lambda$ depends only on $B$ if $A$ is of type (i). 

We proceed as in \cite[Proposition 5]{K13}, and assume by contradiction that \(\lambda = 0\). Then, the corresponding eigenfunction \(\psi\in L^2(\Omega)\) satisfies
\begin{equation}\label{eq:diamag}
0 = \cf^\Omega(\psi) = \int_{\Omega}  \left[|(\dH+iA)\psi|^2 - \frac{r^2}{\rho^4}|\psi|^2 \right]\,dq\ge \int_\Omega |\dH |\psi||^2 - \int_\Omega |\psi|^2\frac{r^2}{\rho^4}\,dq \ge 0,
\end{equation}
where the second estimate is due to the Garofalo--Lanconelli Hardy inequality~\eqref{eq:GL}, and the first one is due to the diamagnetic inequality~\eqref{eq:diamagnetic}.

By (\ref{eq:diamag}) we thus have that \(|\psi|\) coincide with the ground state of \(P_0^\Omega\). Up to restricting to a connected component of $\Omega$, we can always assume $\Omega$ to be connected. Proposition~\ref{prop:simple} then implies that $\rho^{-1}$ is the unique ground state of $P_0^{\Omega}$.  Hence, up to normalization, we can choose  \(\psi = e^{if} \rho^{-1}\) for smooth real-valued $f$. A direct computation yields \(|(\dH+iA) \psi|^2 = |\dH |\psi||^2 + |(A+\dH f)\psi|^2\). Since the inequalities in \eqref{eq:diamag} are now equalities, this immediately yields
\begin{equation}
\int_\Omega |A+\dH f|^2|\psi|^2\,dq = 0.
\end{equation}
Hence, \(A +\dH f \equiv 0\) on $\Omega$. That is, $A$ is gauge equivalent to the null vector potential on $\Omega$, and in particular the associated magnetic field $B=dA$ vanishes on $\Omega$. This contradicts the assumption that $A$ is of type (i) or (ii), and thus completes the proof. 
\end{proof}

\appendix

\section{Xiao's Hardy inequality for Aharonov--Bohm potentials}
\label{app:xiao}
%
Let us prove the following result, which raises a crucial criticism of the result \cite[Thm.~1.1]{Xiao2015} due to Xiao. 
Observe that this is just an instance of gauge invariance and the fact that the first co-homology group of \(\mathbb H^1\setminus\{0\}\) is trivial.

\begin{theorem}
Consider the following magnetic potential on $\mathbb H^1\setminus\{0\}$:
\begin{equation*}
    A := -\frac{Y\rho}{\rho} \,dx + \frac{X\rho}\rho\, dy,
\end{equation*}
where $\rho(\xi,z) := \left( |\xi|^4 + 16z^2\right)^{1/4}$ 
is the Koranyi distance, as above.
Then, for any $\beta\in \mathbb R$, the following Hardy inequality is sharp
\begin{equation}\label{eq:GL-Hardy}
\int_{\mathbb H^1} \frac{|(\dH + i\beta A) u|^2}{|\nabla \rho|^2}\,dp \ge \int_{\mathbb H^1}\frac{|u|^2}{\rho^2}, \qquad \forall u\in C^\infty_c(\mathbb H^1).
\end{equation}
\end{theorem}

\begin{proof}
The fact that \eqref{eq:GL-Hardy} is sharp when $\beta=0$ is well known. 
A proof is given in \cite[Lem.~2.1]{Xiao2015}, or can be derived by adapting the proof of \cite[Prop.~23]{FranceschiPrandi20}.

The case $\beta\neq 0$ follows by showing that $A$ is gauge equivalent to the zero vector potential, thus reducing \eqref{eq:GL-Hardy} to the already established case of $\beta =0$.
Indeed, one can directly check that $A=-dg$, where
\begin{equation}\label{eq:dg}
g(\xi,z) :=
\begin{cases}
\frac12 \arctan\left( \frac{|\xi|^2}{4z} \right) & \text{ if } z>0,\\
\frac{\pi}{4} & \text{ if } z=0,\\
\frac12 \arctan\left( \frac{|\xi|^2}{4z} \right)+\frac{\pi}2 & \text{ if } z<0.
\end{cases}
\end{equation}
It is immediate to observe that \(g\) is continuous on \(\mathbb H^1\setminus\{0\}\), and straightforward computations show that \(Xg = Y\rho/\rho\) and \(Yg = -X\rho/\rho\) on \(\mathbb H^1 \setminus\{z=0\}\). Thus $g$ is  of class $C^1$ on $\mathbb H^1\setminus\{0\}$, completing the proof.
\end{proof}

The argument of proof in \cite{Xiao2015} is flawed. This is due to the implicit assumption in the derivation of relation \cite[(2.9)]{Xiao2015} that the functions $u_n(\rho,\alpha)$ are periodic 
with respect to~$\alpha$, 
which is necessary to guarantee that $u'_{n,k} = i2k u_{n,k}$.
A correct version of the theorem requires  such an assumption:

\begin{theorem}
\label{thm:xiao-correct}
    For all $u\in C^\infty_c(\mathbb H^1)$ such that $u(0,0,z)=u(0,0,-z)$ for all $z\in \mathbb R$, it holds
    \begin{equation*}
        \int_{\mathbb H^1} \frac{|(\dH + i\beta A)u|^2}{|\nH \rho|^2}\,dp \ge \left(1 + d(\beta, \mathbb Z)^2\right) \int_{\mathbb H^1}\frac{|u|^2}{\rho^2}\,dp.
    \end{equation*}
\end{theorem}

We observe that the above result can be interpreted as a result on functions defined on $X = (\bH^1\setminus\{0\})/\sim$, where we let $(0,0,z)\sim (0,0,-z)$ for all $z\ge0$. Then, one can make sense of the exterior differential on $X$ (which is not a manifold) and observe that $A$ is indeed a closed but not exact form.
Namely, the function $g$ defined in \eqref{eq:dg} cannot be modified in order that $g(0,0,z)=g(0,0,-z)$ without changing its differential. 
The following result shows that the only Aharonov--Bohm potentials on $X$ are the ones considered in Theorem~\ref{thm:xiao-correct}. 

\begin{proposition}
    The space $X$ can be deformation retracted to $\mathbb S^1$. In particular, 
    \begin{equation}
        H^1(X, \mathbb R) \simeq H^1_{\mathrm{dR}}(\mathbb S^1) \simeq \mathbb R,
    \end{equation}
    where $H^1(X,\mathbb R)$ denotes the singular cohomology of $X$ with real coefficients.
\end{proposition}

\begin{proof}
    Considering spherical coordinates $(r,\theta,\varphi)\in \mathbb R_+\times \mathbb S^1 \times (-\pi/2,\pi/2)$ on $\mathbb R^3$, one easily verifies that $X \simeq (\mathbb R_+\times \mathbb S^1 \times \mathbb S^1)/\sim$ where, letting $o$ be a fixed point of $\mathbb S^1$, we identify for any $r>0$ all points of the set $\{r\}\times \mathbb S^1\times \{o\}$. Then, it is standard to show that $X$ can be deformation retracted to $X'= (\mathbb S^1 \times \mathbb S^1)/\sim$, where now we identify all points of $\mathbb S^1\times \{o\}$. Finally, $X'$ can be deformation retracted to $\mathbb S^1$ by taking a point $(\theta,\varphi)\in X'$, moving it continuously to $(\theta, o)\simeq (0,o)$, and then bringing it back to $(0,\varphi)\in \{0\}\times \mathbb S^1$.
\end{proof}

%

\subsection*{Acknowledgments}
  B.C.~is member of GNAMPA (INDAM) and he is supported by Fondo Sociale Europeo – Programma
  Operativo Nazio\-nale Ricerca e Innovazione 2014-2020,
  progetto PON: progetto AIM1892920-attivit\`a 2, linea 2.1.
The work of D.K. was partially supported by the EXPRO grant No.\ 20-17749X
of the Czech Science Foundation (GA\v{C}R).

\bibliographystyle{abbrv}
\bibliography{biblio02}

\begin{thebibliography}{10}

\bibitem{ABFP}
R.~Adami, U.~Boscain, V.~Franceschi, and D.~Prandi.
\newblock Point interactions for 3{D} sub-{L}aplacians.
\newblock {\em Ann. Inst. H. Poincar\'{e} Anal. Non Lin\'{e}aire},
  38(4):1095--1113, 2021.

\bibitem{AdamiTeta}
R.~Adami and A.~Teta.
\newblock On the {Aharonov}--{Bohm} {Hamiltonian}.
\newblock {\em Letters in Mathematical Physics}, 43(1):43--54, 1998.

\bibitem{AermarkLaptev11}
L.~Aermark and A.~Laptev.
\newblock Hardy's inequality for the {G}rushin operator with a magnetic field
  of {A}haranov-{B}ohm type.
\newblock {\em Algebra i Analiz}, 23(2):1--8, 2011.

\bibitem{agrachev_barilari_boscain_2019}
A.~Agrachev, D.~Barilari, and U.~Boscain.
\newblock {\em A Comprehensive Introduction to Sub-Riemannian Geometry}.
\newblock Cambridge Studies in Advanced Mathematics. Cambridge University
  Press, 2019.

\bibitem{aharonov1959significance}
Y.~Aharonov and D.~Bohm.
\newblock Significance of electromagnetic potentials in the quantum theory.
\newblock {\em Phys. Rev.}, 115:485--491, Aug 1959.

\bibitem{BCX05}
H.~Bahouri, J.-Y. Chemin, and C.-J. Xu.
\newblock Trace and trace lifting theorems in weighted {S}obolev spaces.
\newblock {\em J. Inst. Math. Jussieu}, 4(4):509--552, 2005.

\bibitem{brezis1997blow}
H.~Brezis and J.~L. V{\'a}zquez.
\newblock Blow-up solutions of some nonlinear elliptic problems.
\newblock {\em Rev. Mat. Univ. Complut. Madrid}, 10(2):443--469, 1997.

\bibitem{cassano2018self}
B.~Cassano and F.~Pizzichillo.
\newblock Self-adjoint extensions for the {Dirac} operator with {Coulomb}-type
  spherically symmetric potentials.
\newblock {\em Letters in Mathematical Physics}, 108(12):2635--2667, 2018.

\bibitem{Cazacu2016}
C.~Cazacu and D.~Krej{\v{c}}iř{\'{i}}k.
\newblock {The Hardy inequality and the heat equation with magnetic field in
  any dimension}.
\newblock {\em Communications in Partial Differential Equations},
  41(7):1056--1088, 2016.

\bibitem{CiattiCowlingRicci15}
P.~Ciatti, M.~G. Cowling, and F.~Ricci.
\newblock Hardy and uncertainty inequalities on stratified {L}ie groups.
\newblock {\em Adv. Math.}, 277:365--387, 2015.

\bibitem{CiattiRicciSundari07}
P.~Ciatti, F.~Ricci, and M.~Sundari.
\newblock Heisenberg-{P}auli-{W}eyl uncertainty inequalities and polynomial
  volume growth.
\newblock {\em Adv. Math.}, 215(2):616--625, 2007.

\bibitem{D'Ambrosio2003}
L.~D'Ambrosio.
\newblock Hardy inequalities related to {G}rushin type operators.
\newblock {\em Proc. Amer. Math. Soc.}, 132(3):725--734, 2004.

\bibitem{D'Ambrosio2004}
L.~D'Ambrosio.
\newblock Some {H}ardy inequalities on the {H}eisenberg group.
\newblock {\em Differ. Uravn.}, 40(4):509--521, 575, 2004.

\bibitem{D'Ambrosio2005}
L.~D'Ambrosio.
\newblock Hardy-type inequalities related to degenerate elliptic differential
  operators.
\newblock {\em Ann. Sc. Norm. Super. Pisa Cl. Sci. (5)}, 4(3):451--486, 2005.

\bibitem{DGP2011}
D.~Danielli, N.~Garofalo, and N.~C. Phuc.
\newblock Hardy-{S}obolev type inequalities with sharp constants in
  {C}arnot-{C}arath\'{e}odory spaces.
\newblock {\em Potential Anal.}, 34(3):223--242, 2011.

\bibitem{dabrowski1998aharonov}
L.~D\c{a}browski and P.~\v{S}\v{t}ov\'{\i}\v{c}ek.
\newblock Aharonov-{B}ohm effect with {$\delta$}-type interaction.
\newblock {\em J. Math. Phys.}, 39(1):47--62, 1998.

\bibitem{edmunds2018spectral}
D.~E. Edmunds and W.~D. Evans.
\newblock {\em Spectral theory and differential operators}.
\newblock Oxford University Press, 2018.

\bibitem{Escobedo1987}
M.~Escobedo and O.~Kavian.
\newblock {Variational problems related to self-similar solutions of the heat
  equation}.
\newblock {\em Nonlinear Analysis, Theory, Methods and Applications},
  11(10):1103--1113, 1987.

\bibitem{Evans-Zett_1978}
W.~D. Evans and A.~Zettl.
\newblock Dirichlet and separation results for {S}chr{\"o}dinger-type
  operators.
\newblock {\em Proc. Roy. Soc. Edinburgh Sect. A}, 80:151--162, 1978.

\bibitem{FKLV20}
L.~Fanelli, D.~Krej\v{c}i\v{r}\'{\i}k, A.~Laptev, and L.~Vega.
\newblock On the improvement of the {H}ardy inequality due to singular magnetic
  fields.
\newblock {\em Comm. Partial Differential Equations}, 45(9):1202--1212, 2020.

\bibitem{Folland73}
G.~B. Folland.
\newblock A fundamental solution for a subelliptic operator.
\newblock {\em Bull. Amer. Math. Soc.}, 79:373--376, 1973.

\bibitem{follandEstimates1974}
G.~B. Folland and E.~M. Stein.
\newblock Estimates for the {$\bar{\partial}_b$} complex and analysis on the
  {{Heisenberg}} group.
\newblock {\em Communications on Pure and Applied Mathematics}, 27(4):429--522,
  July 1974.

\bibitem{FranceschiPrandi20}
V.~Franceschi and D.~Prandi.
\newblock Hardy-{{Type Inequalities}} for the
  {{Carnot}}\textendash{{Carath\'eodory Distance}} in the {{Heisenberg Group}}.
\newblock {\em The Journal of Geometric Analysis}, 31(3):2455--2480, Mar. 2021.

\bibitem{FPR20}
V.~Franceschi, D.~Prandi, and L.~Rizzi.
\newblock On the essential self-adjointness of singular sub-{L}aplacians.
\newblock {\em Potential Anal.}, 53(1):89--112, 2020.

\bibitem{FOV13}
B.~Franchi, E.~Obrecht, and E.~Vecchi.
\newblock On a class of semilinear evolution equations for vector potentials
  associated with {M}axwell's equations in {C}arnot groups.
\newblock {\em Nonlinear Anal.}, 90:56--69, 2013.

\bibitem{FranchiTesi12}
B.~Franchi and M.~C. Tesi.
\newblock Wave and {M}axwell's equations in {C}arnot groups.
\newblock {\em Commun. Contemp. Math.}, 14(5):1250032, 62, 2012.

\bibitem{GarofaloLanconelli1990}
N.~Garofalo and E.~Lanconelli.
\newblock {Frequency functions on the Heisenberg group, the uncertainty
  principle and unique continuation}.
\newblock {\em Annales de l'institut Fourier}, 40(2):313--356, 1990.

\bibitem{GoldsteinKombe2008}
J.~A. Goldstein and I.~Kombe.
\newblock The {H}ardy inequality and nonlinear parabolic equations on {C}arnot
  groups.
\newblock {\em Nonlinear Anal.}, 69(12):4643--4653, 2008.

\bibitem{GKY2018}
J.~A. Goldstein, I.~Kombe, and A.~Yener.
\newblock A unified approach to weighted {H}ardy type inequalities on {C}arnot
  groups.
\newblock {\em Discrete Contin. Dyn. Syst.}, 37(4):2009--2021, 2017.

\bibitem{HK2000}
P.~Haj{\l}asz and P.~Koskela.
\newblock Sobolev met {P}oincar\'{e}.
\newblock {\em Mem. Amer. Math. Soc.}, 145(688):x+101, 2000.

\bibitem{Helffer88}
B.~Helffer.
\newblock {\em Semi-classical analysis for the {S}chr\"{o}dinger operator and
  applications}, volume 1336 of {\em Lecture Notes in Mathematics}.
\newblock Springer-Verlag, Berlin, 1988.

\bibitem{Helffer97-ln}
B.~Helffer.
\newblock Semiclassical analysis for the {S}chr\"{o}dinger operator with
  magnetic wells (after {R}. {M}ontgomery, {B}. {H}elffer-{A}. {M}ohamed).
\newblock In {\em Quasiclassical methods ({M}inneapolis, {MN}, 1995)},
  volume~95 of {\em IMA Vol. Math. Appl.}, pages 99--114. Springer, New York,
  1997.

\bibitem{hormanderHypoelliptic1967a}
L.~H{\"o}rmander.
\newblock Hypoelliptic second order differential equations.
\newblock {\em Acta Mathematica}, 119(0):147--171, 1967.

\bibitem{JL88}
D.~Jerison and J.~M. Lee.
\newblock Extremals for the {S}obolev inequality on the {H}eisenberg group and
  the {CR} {Y}amabe problem.
\newblock {\em J. Amer. Math. Soc.}, 1(1):1--13, 1988.

\bibitem{Kato}
T.~Kato.
\newblock {\em Perturbation theory for linear operators}.
\newblock Classics in Mathematics. Springer-Verlag, Berlin, 1995.
\newblock Reprint of the 1980 edition.

\bibitem{Kogoj2016}
A.~E. Kogoj and S.~Sonner.
\newblock Hardy type inequalities for {$\Delta_{\lambda}$}-{L}aplacians.
\newblock {\em Complex Var. Elliptic Equ.}, 61(3):422--442, 2016.

\bibitem{Krejcirik2013}
D.~Krej{\v{c}}iř{\'{i}}k.
\newblock {The improved decay rate for the heat semigroup with local magnetic
  field in the plane}.
\newblock {\em Calculus of Variations and Partial Differential Equations},
  47(1-2):207--226, 2013.

\bibitem{K13}
D.~Krej\v{c}i\v{r}\'{\i}k.
\newblock The improved decay rate for the heat semigroup with local magnetic
  field in the plane.
\newblock {\em Calc. Var. Partial Differential Equations}, 47(1-2):207--226,
  2013.

\bibitem{LW99}
A.~Laptev and T.~Weidl.
\newblock Hardy inequalities for magnetic {D}irichlet forms.
\newblock In {\em Mathematical results in quantum mechanics ({P}rague, 1998)},
  volume 108 of {\em Oper. Theory Adv. Appl.}, pages 299--305. Birkh\"{a}user,
  Basel, 1999.

\bibitem{LiebLoss}
E.~H. Lieb and M.~Loss.
\newblock {\em Analysis}, volume~14 of {\em Graduate Studies in Mathematics}.
\newblock American Mathematical Society, Providence, RI, second edition, 2001.

\bibitem{Loiudice18}
A.~Loiudice.
\newblock Local behavior of solutions to subelliptic problems with {H}ardy
  potential on {C}arnot groups.
\newblock {\em Mediterr. J. Math.}, 15(3):Paper No. 81, 20, 2018.

\bibitem{Mohamed-Raikov}
A.~Mohamed and G.~D. Raikov.
\newblock On the spectral theory of the {S}chr{\"o}dinger operator with
  electromagnetic potential.
\newblock In {\em Pseudo-Differential Calculus and Mathematical Physics},
  volume~5 of {\em Adv. Part. Diff. Equat.}, pages 298--390. Akademie-Verlag,
  1994.

\bibitem{Montgomery1995}
R.~Montgomery.
\newblock {Hearing the zero locus of a magnetic field}.
\newblock {\em Communications in Mathematical Physics}, 168(3):651--675, 1995.

\bibitem{montiRearrangements2014}
R.~Monti.
\newblock Rearrangements in metric spaces and in the {H}eisenberg group.
\newblock {\em J. Geom. Anal.}, 24(4):1673--1715, 2014.

\bibitem{NZY2001}
P.~Niu, H.~Zhang, and Y.~Wang.
\newblock Hardy type and {R}ellich type inequalities on the {H}eisenberg group.
\newblock {\em Proc. Amer. Math. Soc.}, 129(12):3623--3630, 2001.

\bibitem{pankrashkin2011spectral}
K.~Pankrashkin and S.~Richard.
\newblock Spectral and scattering theory for the {Aharonov}--{Bohm} operators.
\newblock {\em Reviews in Mathematical Physics}, 23(01):53--81, 2011.

\bibitem{Pinchover_2007}
Y.~Pinchover.
\newblock Topics in the theory of positive solutions of second-order elliptic
  and parabolic partial differential equations.
\newblock In {F.~Gesztesy, et al.}, editor, {\em Spectral Theory and
  Mathematical Physics: A Festschrift in Honor of Barry Simon's 60th Birthday},
  volume~76 of {\em Proc. Sympos. Pure Math.}, pages 329--356. Amer. Math.
  Soc., Providence, RI, 2007.

\bibitem{RS4}
M.~Reed and B.~Simon.
\newblock {\em Methods of modern mathematical physics. {IV}. {A}nalysis of
  operators}.
\newblock Academic Press [Harcourt Brace Jovanovich, Publishers], New
  York-London, 1978.

\bibitem{Rumin1994}
M.~Rumin.
\newblock {Formes diff{\'{e}}rentielles sur les vari{\'{e}}t{\'{e}}s de
  contact}.
\newblock {\em Journal of Differential Geometry}, 39(2):281--330, 1994.

\bibitem{RS17b}
M.~Ruzhansky and D.~Suragan.
\newblock Layer potentials, {K}ac's problem, and refined {H}ardy inequality on
  homogeneous {C}arnot groups.
\newblock {\em Adv. Math.}, 308:483--528, 2017.

\bibitem{Ruzhansky2017}
M.~Ruzhansky and D.~Suragan.
\newblock {On horizontal Hardy, Rellich, Caffarelli–Kohn–Nirenberg and
  p-sub-Laplacian inequalities on stratified groups}.
\newblock {\em Journal of Differential Equations}, 262(3):1799--1821, 2017.

\bibitem{Weidl_1999}
T.~Weidl.
\newblock A remark on {H}ardy type inequalities for critical {S}chr{\"o}dinger
  operators with magnetic fields.
\newblock {\em Oper. Theory Adv. Appl.}, 110:345--352, 1999.

\bibitem{Weidl_1999a}
T.~Weidl.
\newblock Remarks on virtual bound states for semi-bounded operators.
\newblock {\em Commun. in Partial Differential Equations}, 24:25--60, 1999.

\bibitem{Xiao2015}
Y.~Xiao.
\newblock Hardy inequalities with {A}haronov-{B}ohm type magnetic field on the
  {H}eisenberg group.
\newblock {\em J. Inequal. Appl.}, 2015(95), 2015.

\end{thebibliography}
\end{document}